\documentclass[10pt,preprint]{elsarticle}

\makeatletter
\def\ps@pprintTitle{%
 \let\@oddhead\@empty
 \let\@evenhead\@empty
 \def\@oddfoot{\hfil\it\today}%
 \let\@evenfoot\@oddfoot}
\makeatother

\usepackage{xspace,exscale}
\usepackage{fancybox}
\usepackage{graphicx}
\usepackage{color}
\usepackage{amsfonts}
\usepackage{amsthm}
\usepackage{amssymb}
\usepackage{url}

\usepackage{amsmath}
	\makeatletter
	\let\over=\@@over \let\overwithdelims=\@@overwithdelims
	\let\atop=\@@atop \let\atopwithdelims=\@@atopwithdelims
  	\let\above=\@@above \let\abovewithdelims=\@@abovewithdelims
  	\makeatother
\usepackage{fixltx2e}

\usepackage{rotating}

\usepackage{ifpdf}

\usepackage{subfigure}
\usepackage{psfrag}

\usepackage[
            CJKbookmarks=true,
            bookmarksnumbered=true,
            bookmarksopen=true,
            colorlinks=true,
            citecolor=red,
            linkcolor=blue,
            anchorcolor=red,
            urlcolor=blue,
            pdfauthor={Polyanskiy and Samorodnitsky}
            ]{hyperref}

\newcommand{\matc}{\ensuremath{\mathcal{C}}}
\newcommand{\matx}{\ensuremath{\mathcal{X}}}

\newcommand{\matn}{\ensuremath{\mathcal{N}}}

\newcommand{\mate}{\ensuremath{\mathcal{E}}}

\newcommand{\Ent}{\ensuremath{\mathrm{Ent}}}

\newcommand{\mreals}{\ensuremath{\mathbb{R}}}
\newcommand{\mcomplex}{\ensuremath{\mathbb{C}}}

\newcommand{\FF}{\ensuremath{\mathbb{F}}}
\newcommand{\ZZ}{\ensuremath{\mathbb{Z}}}

\ifx\eqref\undefined
	\newcommand{\eqref}[1]{~(\ref{#1})}
\fi
\ifx\mod\undefined
	\def\mod{\mathop{\rm mod}}
\fi

\def\supp{\mathop{\rm supp}}
\def\vol{\mathop{\rm vol}}

\def\tr{\mathop{\rm tr}}

\def\sign{\mathop{\rm sgn}}

\def\EE{\mathbb{E}\,}

\def\Var{\mathrm{Var}}

\def\PP{\mathbb{P}}

\def\eqdef{\stackrel{\triangle}{=}}

\def\unifto{\mathop{{\mskip 3mu plus 2mu minus 1mu%
	\setbox0=\hbox{$\mathchar"3221$}%
	\raise.6ex\copy0\kern-\wd0%
	\lower0.5ex\hbox{$\mathchar"3221$}}\mskip 3mu plus 2mu minus 1mu}}

\def\Bern{\mathrm{Bern}}

\ifx\lesssim\undefined
\def\simleq{{{\mskip 3mu plus 2mu minus 1mu%
	\setbox0=\hbox{$\mathchar"013C$}%
	\raise.2ex\copy0\kern-\wd0%
	\lower0.9ex\hbox{$\mathchar"0218$}}\mskip 3mu plus 2mu minus 1mu}}
\else
\def\simleq{\lesssim}
\fi

\ifx\gtrsim\undefined
\def\simgeq{{{\mskip 3mu plus 2mu minus 1mu%
	\setbox0=\hbox{$\mathchar"013E$}%
	\raise.2ex\copy0\kern-\wd0%
	\lower0.9ex\hbox{$\mathchar"0218$}}\mskip 3mu plus 2mu minus 1mu}}
\else
\def\simgeq{\gtrsim}
\fi

\def\la{\langle}
\def\ra{\rangle}

\newif\ifmapx
{\catcode`/=0 \catcode`\\=12/gdef/mkillslash\#1{#1}}
\edef\jobnametmp{\expandafter\string\csname uncertain_apx\endcsname}
\edef\jobnameapx{\expandafter\mkillslash\jobnametmp}
\edef\jobnameexpand{\jobname}
\ifx\jobnameexpand\jobnameapx
\mapxtrue
\else
\mapxfalse
\fi

\long\def\apxonly#1{\ifmapx{\color{blue}#1}\fi}

\newtheorem{theorem}{Theorem}
\newtheorem{lemma}[theorem]{Lemma}
\newtheorem{corollary}[theorem]{Corollary}
\newtheorem{definition}{Definition}
\newtheorem{proposition}[theorem]{Proposition}
\newtheorem{remark}{Remark}

\begin{document}

\title{Improved log-Sobolev inequalities, hypercontractivity and uncertainty principle on the hypercube}

\author[yp]{Yury Polyanskiy}
\ead{yp@mit.edu}
\address[yp]{Department of Electrical Engineering
and Computer Science, MIT, Cambridge, MA 02139 USA.}
\author[as]{Alex Samorodnitsky}
\ead{salex@cs.huji.ac.il}
\address[as]{School of Engineering and Computer Science,
The Hebrew University of Jerusalem,
Jerusalem 91904, Israel.}
\begin{abstract}
Log-Sobolev inequalities (LSIs) upper-bound entropy via a multiple of the Dirichlet form (i.e. norm of a gradient).
In this paper we prove a family of entropy-energy inequalities for the binary hypercube which provide a non-linear comparison between the entropy and the Dirichlet form and improve on the usual LSIs for functions with small support.
These non-linear LSIs, in turn, imply a new version of the hypercontractivity
for such functions. As another consequence, we derive a sharp form of the uncertainty principle for the hypercube:
a function whose energy is concentrated on a set of small size, and whose Fourier energy is concentrated on a small
Hamming ball must be zero. The tradeoff between the sizes that we derive is asymptotically optimal.
This new uncertainty principle implies a new estimate on the size of Fourier coefficients of sparse Boolean functions. 
We observe that an analogous  (asymptotically optimal) uncertainty principle in the Euclidean space follows from the
sharp form of Young's inequality due to Beckner. This hints that non-linear LSIs augment Young's
inequality (which itself is sharp for finite groups).
\end{abstract}

\maketitle
\tableofcontents

\section{Introduction}

\subsection{Definitions, background}
We introduce some standard notions for continous-time semigroups on finite state spaces, e.g.~\cite[Section
1.7.1]{bakry2006functional}.
Consider a finite alphabet $\matx$ and a matrix $(L_{x,y})_{x,y \in \matx}$ such that 1) $L_{x,y}\ge0$ for $x\neq y$;
and 2) $\sum_{y\in\matx} L_{x,y}=0$ for all $x$. Then $T_t = e^{tL}$ is a stochastic semigroup, for which we assume that
$\pi$ is a stationary measure. We define $\|f\|_p \eqdef \EE^{1\over p}[|f|^p]$ and $(f,g)=\EE[fg]$ with expectation over $\pi$. The
Dirichlet form of semigroup $T_t$ is
$$ \mate(f,g) \eqdef -\sum_{x,y} L_{x,y}f(y) g(x) \pi(x) = \EE_{\pi}[(-Lf)g]\,.$$
We also define $T_t^{\otimes n}$ -- a product semigroup on $\matx^n$ -- and notice that its Dirichlet form is given
by
\begin{equation}\label{eq:maten}
	\mate_n(f,g) \eqdef \sum_{k=1}^n \sum_{x_{\hat k} \in \matx^{n-1}} h(x_{\hat k}) \prod_{j\neq k} \pi(x_j)\,,
\end{equation}
where $x_{\hat k}=(x_1,\ldots,x_{k-1},x_{k+1},\ldots,x_n)$ and $h(x_{\hat k}) = \mate(f(x_{\hat k},\cdot), g(x_{\hat
k}, \cdot))$ is the action of Dirichlet form $\mate$ on $k$-th coordinate of $f$ and $g$ with other coordinates held
frozen.

We will be interested in understanding evolution of $\|f_t\|_p$, where $f_t = T_t f$. Notice that derivative of this quantity in $t$ gives
rise to $\mate(f_t, f_t^{p-1})$, whereas derivative in $p$ leads to $\Ent_{\pi}(|f_t|^p)$, where for any $g\ge 0$ we
define
$$ \Ent_{\pi}(g) \eqdef \EE_{\pi}\left[g(X) \ln{g\over \EE[g]}\right] = \EE[g]  D(\pi^{(g)}\|\pi)\,,$$
with $\pi^{(g)}(x) \eqdef {g(x)\pi(x)\over \EE[g]}$ and $D(\cdot\|\cdot)$ -- the Kullback-Leibler divergence.
The idea of bounding these two derivatives (in $t$ and $p$) in terms of one another was introduced in ~\cite{LG75}.
This explains introduction of the following concept.

We say that a semigroup admits a $p$-logarithmic Sobolev inequality ($p$-LSI for short), see~\cite[Section
3]{bakry1994hypercontractivite}, if for some constant $\alpha_p$
\begin{equation}\label{eq:plsi}
	\Ent_{\pi}(f^p) \le {1\over \alpha_p} \mate(f, f^{p-1})\,,
\end{equation}

We note that $\mate(f,f^{p-1}) \ge 0$ for $p>1$ and $\mate(f,f^{p-1})\le0$ for $p<1$ and this implies corresponding signs
for constants $\alpha_p$. As $p\to1$ we have $\mate(f,f^{p-1})\to0$ and so we need to renormalize by $1\over p-1$ in
this limit. Consequently, we define $1$-LSI as
$$ \Ent_{\pi}(f) \le {1\over \alpha_1} \mate(f, \ln f)\,,$$
which is required to hold for all $f>0$ on $\matx$.

 We do not
discuss full history of LSI, only mentioning that $p=2$ case originated in~\cite{LG75} and $p=1$
in~\cite{bobkov1998modified}; for more detailed history see~\cite{Bobkov2006,MOS12}. The $p=1$ case is also known as
modified LSI and connects to (one version of) discrete Ricci curvature~\cite{erbar2012ricci}.

We will mostly deal in this paper with a special case of a hypercube. Namely, we set
$\matx=\{0,1\}$, $L_{x,y}=-1\{x=y\} + 1/2$, $\pi=\Bern(1/2)$
and
\begin{equation}\label{eq:hypercube}
	T_t f(x) = f(x) {1+e^{-t}\over 2} + f(1-x){1-e^{-t}\over 2}\,.
\end{equation}
For this case the best LSI constants are $\alpha_p={2(p-1)\over p^2}$, see~\cite[Theorem 2.2.8]{saloff1997lectures}, and the Dirichlet form takes particularly simple
form:
\begin{align}\label{eq:diri_cube}
	\mate_n(f,g) &= -{1\over 2}(\triangle f,g), \qquad \triangle f(x) \eqdef \sum_{y: y\sim x} (f(y)-f(x))\\
	\mate_n(f,f) &= {1\over 4} 2^{-n} \sum_{(x,y): x\sim y} (f(x)-f(y))^2 \,,
\end{align}
where $x\sim y$ means that $x,y\in\{0,1\}^n$ differ in precisely one coordinate. \apxonly{In the second summation we
sum over all unordered pairs $(x,y)$ (i.e. $n2^n$ terms). Other forms:
\begin{align*} \nabla f(x) &\eqdef (f(x)-f(x+e_1), \ldots, f(x) - f(x+e_n))\,,\\
	\mate_n(f,g) &= {1\over 4}(\nabla f, \nabla g)
\end{align*}
}

We also quote one inequality from information theory, known as Mrs. Gerber's lemma, or MGL~\cite{wyner1973theorem},
which we write in the following form: for any $f\ge 0$
\begin{equation}\label{eq:mgl_1}
	{1\over n}{\Ent_{\pi^n}(T_t^{\otimes n} f)\over \EE[f]} \le \ln 2 -
		m\left(t, {1\over n} {\Ent_{\pi^n}(f)\over \EE[f]}\right)\,.
\end{equation}
Here $m(t,x) = h(h^{-1}(\ln2 - x) * {1-e^{-t}\over 2})$, where
\begin{align}\label{eq:hbin_def}
	h(x) \eqdef -x\ln x - (1-x)\ln(1-x)
\end{align}is
the binary entropy function, $h^{-1}:[0,\ln 2]\to[0,1/2]$ is its functional inverse and $a*b=(1-a)b+(1-b)a$ is the binary
convolution.

A less cryptic restatement of MGL is the following: For all $t\ge 0$
\begin{equation}\label{eq:mgl_2}
	\Ent_{\pi^n}(T_t^{\otimes n} f) \le \Ent_{\pi^n}(T_t^{\otimes n} f_{iid})\,,
\end{equation}
where $f_{iid}(x)=\prod_{k=1}^n f_1(x_k)$ with $f_1(\cdot)$ selected so that a) in~\eqref{eq:mgl_2} the equality holds for $t=0$; b)
$\EE[f_{iid}] = \EE[f]$. In other
words, MGL states that among all functions $f$ on the hypercube, $\Ent$ decreases slowest
for product functions. (Note that for a general product semigroup the statement~\eqref{eq:mgl_2} does not need to
hold even if we add an extra constraint that $f_1$ should be chosen so that, in addition to a) and b), it maximizes
$\Ent_\pi(T_t f_1)$; see~\cite[Theorem 6]{witsenhausen1974entropy}.)

\subsection{Motivation and Organization}

We motivate our investigation by the following three questions:
\begin{itemize}
\item Log-Sobolev inequality implies an estimate of the form
$$ \Ent(T_t^{\otimes n} f) \le e^{-Ct} \Ent(f)\,. $$
However, for the hypercube a stronger estimate is given by the MGL~\eqref{eq:mgl_1}.
\textit{Can MGL be derived from some strengthening of LSI?}

Note that by a method of comparison of Dirichlet forms, results derived from log-Sobolev inequalities can then be extended to
semigroups other than $T_t^{\otimes n}$. As an illustration, note that~\cite[Example 3.3]{DSC96} estimates speed of
convergence of a Metropolis chain on $\{0,\ldots n\}$ by comparing to $T_t^{\otimes n}$. Our methods allow to show
better estimates, similar to~\eqref{eq:mgl_1}.

\item Hypercontractivity inequality for the hypercube (variously attributed to~\cite{EN66,AB70,WB75,LG75}) says
\begin{equation}\label{eq:bonami}
	 \|T_t^{\otimes n} f\|_{p(t)} \le \|f\|_{p_0}\,, \qquad p(t) = 1+ (p_0-1)e^{2t}, \, p_0\ge1\,.
\end{equation}
	This is well known to be tight in the sense that for any $q>p(t)$ we can find $f$ s.t. $\|T_t f\|_q >
	\|f\|_{p_0}$. However, such $f$ will be very close to identity (for this particular
	semigroup). \textit{Is it possible to improve the range of $(p,q)$ in~\eqref{eq:bonami} provided $f$ is far from
	identity (say in the sense of $|\supp f|\ll 2^n$)?}

	For example, it is clear that $\|T_t^{\otimes n}\|_{1\to\infty} = (1+e^{-t})^n$. If $f$ has small support, we
	have $\|f\|_{p_0} \ge e^{n\rho_0} \|f\|_1$, where $\rho_0=(1-{1\over p_0}){1\over n}\ln {2^n\over |\supp f|}$ and thus
	\begin{equation}\label{eq:hca_p}
		\|T_t^{\otimes n} f\|_{\infty} \le \|f\|_{p_0}, \qquad \forall t \ge \ln {1\over e^{\rho_0}-1}
	\end{equation}	
	which is a significant improvement of~\eqref{eq:bonami} for large times $t$.
\item Finally, it was noticed in~\cite{beckner1995pitt} that LSIs on Euclidean space are closely
related to a form of uncertainty principle, which connects the tail behavior of the function and its Fourier transform.
\textit{We ask whether LSIs on finite groups (e.g. hypercube) imply bounds on the tradeoff between the sizes of
supports of the function and its Fourier image.}
\end{itemize}

All these questions will be answered positively.

The structure of the paper is the following. In Section~\ref{sec:lsi} we describe the main concept of this paper -- the
non-linear LSIs and prove some of its consequences, such as refined hypercontractivity and general MGL. In
Section~\ref{sec:hyper} we switch from general theory to the particular case of the hypercube. We establish explicit forms
of new LSIs and new hypercontractive estimates for functions of small support. In Section~\ref{sec:up} we apply
the latter to establish a sharp version of the uncertainty principle on the hypercube. Finally, Section~\ref{sec:ft_coding} applies the uncertainty principle
to derive a lower bound on large-degree Fourier coefficients of sparse Boolean functions.

\section{Non-linear log-Sobolev inequalities}\label{sec:lsi}

In this section we introduce a family of non-linear log-Sobolev inequalities (LSI) and prove three implications
relevant for this paper. We mention that special case of $p=2$ (which is the main case, especially for diffusion
semigroups) has been known in analysis for a long time under the name of ``entropy-energy'' inequalities (see below), and thus our
generalization is to consider general $p$. The results we prove are: tensorization
(i.e. extension from $T_t$ to $T_t^{\otimes n}$), integrating $1$-LSI to get entropy decay, integrating $p$-LSI to get
hypercontractivity. The first two are routine verifications, whereas the third required some new ideas.

\begin{definition}\label{def:nlsi} For $p\ge1$ and a \underline{concave}, continuous, non-negative function $\Phi_p:[0,\infty)\to[0,\infty)$ with $\Phi_p(0)=0$,
let us define a $(p,\Phi_p)$-LSI as
\begin{equation}\label{eq:pnlsi}
	{\Ent(f^p)\over \EE[f^p]} \le \Phi_p\left({\mate(f, f^{p-1})\over \EE[f^p]}\right)\,,
\end{equation}
where for $p=1$ we understand $\mate(f, f^{p-1}) = \mate(f, \ln f)$.
For $p<1$ the domain of $\Phi_p$ is replaced with $(-\infty,0]$, and the definition remains the same.
 When convenient, we will restate $(p,\Phi_p)$-LSI in the form
\begin{equation}\label{eq:pnlsi_2}
	\pm {\mate(f, f^{p-1})\over \EE[f^p]} \ge b_p\left({\Ent(f^p)\over \EE[f^p]}\right)\,, \qquad
\end{equation}
where $b_p:[0,\infty)\to[0,\infty]$ is a convex increasing with $b_p(0)=0$ function defined as $b_p(y)\eqdef
	\inf\{x: \Phi_p(x) \ge y\}$ (with the usual agreement that $\inf{\emptyset}=\infty$). 
The $\pm$ is taken to be $+$ for $p\ge 1$ and $-$ for $p<1$.
\end{definition}
\begin{remark} For convenience we define $\Phi_p$ and $b_p$ on $[0,\infty)$ even though the arguments in~\eqref{eq:pnlsi}
and~\eqref{eq:pnlsi_2} may be universally bounded by constants $<\infty$. Note also that a concave non-negative function
on $[0,\infty)$ must be increasing on $[0,a)$ and then constant on $[a,\infty)$ (either interval could be empty).\end{remark}

It is clear from concavity of $\Phi_p$ that the linear-LSIs~\eqref{eq:plsi} are obtained by taking ${1\over\alpha_p} =
\left.{d\over dx}\right|_{x=0}\Phi_p$.
We briefly review the history of such inequalities:
\begin{itemize}
\item For a Lebesgue measure on $\mreals^n$ and $\mate(f,g)=\int(\nabla f, \nabla g)$ the $p=2$ inequality takes the
form:
\begin{equation}\label{eq:eucl_nlsi}
	\int_{\mreals^n} h^2(x) \ln h^2 dx \le {n\over 2} \ln\left({2\over n\pi e} \int_{\mreals^n} \|\nabla h(x)\|^2
	dx\right), \qquad \int h(x)^2 dx = 1\,.
\end{equation}
It appeared in information theory~\cite[(2.3)]{stam1959some} and~\cite{costa1984similarity} as a solution to the problem of minimizing
Fisher information subject to differential entropy constraint (the minimizer is Gaussian density). In
analysis,~\eqref{eq:eucl_nlsi} has been used early by~\cite{weissler1978logarithmic}.
\item Inequality~\eqref{eq:eucl_nlsi} is in fact equivalent to a 2-LSI~\cite{LG75} for Ornstein-Uhlenbeck semigroup:
\begin{equation}\label{eq:lsi_gross}
		\Ent_{\gamma}(f^2) \le 2 \int_{\mreals^n} \|\nabla f\|^2 d\gamma\,,
\end{equation}	
	where $\gamma = \matn(0,I_n)$ (to see equivalence, take $f^2(x) = \lambda h^2(\lambda x) (2\pi)^{n/2}
	e^{\|x\|^2/2}$ with $\lambda = (4 \int \|\nabla h\|^2)^{-{1\over 2}}$ and integrate by parts). It is
	known~\cite{Carlen91} that~\eqref{eq:eucl_nlsi} (resp.,~\eqref{eq:lsi_gross}) is saturated by and only by
	Gaussian densities (resp., exponentials). In particular, taking $f=e^{ax - a^2}$ in~\eqref{eq:lsi_gross} shows
	that for Ornstein-Uhlenbeck semigroup no
	improvement of~\eqref{eq:lsi_gross}, in the sense of Def.~\ref{def:nlsi} is possible (linear LSI is the best
	one).
\item More generally, the $p=2$ inequalities were introduced into operator theory by Davies and
Simon~\cite{davies1984ultracontractivity} under the name of entropy-energy inequalities;
see~\cite{bakry1994hypercontractivite} for a survey.
\item A $p=2$ inequality for the hypercube was proved in~\cite{samorodnitsky2008modified} for the purpose
of showing that the Faber-Krahn problem on the hypercube is asymptotically solved by a Hamming ball. Same reference
mentioned~\cite[paragraph after (11)]{samorodnitsky2008modified} a tightening of hypercontractivity~\eqref{eq:bonami} for
$p_0=2$ and functions of large entropy, although no proof was published at the time.

\item Miclo~\cite{miclo1999majoration} proved a class of restricted entropy-energy inequalities: Given
continuous $\phi:\mreals_+\to\mreals_+$ such that $\phi(x)/(x\ln x)$ is monotonically increasing for large enough $x$ there exists
a continuous increasing $\psi$ and a universal constant $C>0$ such that
	\begin{equation}\label{eq:miclo_lsi}
		\Var_{\pi}[f^2] \ge C \psi(\Ent_{\pi}(f^2))
\end{equation}	
	for all $\pi$ simultaneously but only for functions $f$ satisfying $\EE_\pi[f^2] = 1$ and $\EE_\pi[\phi(f^2)] \le K$
	Function $\psi$ in~\eqref{eq:miclo_lsi} depends on $\phi$ and $K$ roughly via $\psi(8K {x\ln x \over \phi(x)})={x\over
	\phi(x)}$, and in particular $\psi(t)=o(t)$ as $t\to 0$, so that~\eqref{eq:miclo_lsi} does not imply standard
	LSI~\eqref{eq:plsi}. Here, we are interested in improving upon~\eqref{eq:plsi} and also in
	unrestricted inequalities (without constraint on
	$\EE_\pi[\phi(f^2)]$), but for a fixed known $\pi$. \apxonly{One example of inequality for $\phi(x) = x^2$ is:
	  $$ \EE[f^4] \Var[f] \ln {1\over \Var[f]} \ge C \Ent(f^2) \qquad \forall \EE[f^2]=1 $$}
\end{itemize}

We move on to proving general results about non-linear LSIs.

\begin{theorem}\label{th:tenso}(Tensorization)
Suppose that $(p,\Phi_p)$-LSI holds for semigroup $(\matx, \pi, T_t, \mate)$. Then for all $n\ge 1$ the
$(p, n\Phi_p({1\over n}\cdot))$-LSI holds for $(\matx^n, \pi^n, T_t^{\otimes n}, \mate_n)$. In other
words, for all $f:\matx^n\to\mreals_+$  we have
\begin{equation}\label{eq:tenso}
	{1\over n} {\Ent_{\pi^n}(f^p)\over \EE_{\pi^n}[f^p]} \le
	\Phi_p\left({1\over n}{\mate_n(f, f^{p-1})\over \EE_{\pi^n}[f^p]}\right)\,,
\end{equation}	
where $\pi^n = \prod_{k=1}^n \pi$ -- a product measure on $\matx^n$ and $\mate_n$ is the Dirichlet form associated to
the product semigroup~\eqref{eq:maten}.
\end{theorem}

\begin{theorem}\label{th:mgl}(General MGL) Suppose a semigroup $T_t$ admits a $(1, \Phi_1)$-LSI. Let $b_1=\Phi_1^{-1}$
be a convex, strictly increasing inverse of $\Phi_1$ and assume that the differential equation
$$ {d\over dt} \tilde \rho(t) = -b_1(\tilde \rho(t))$$
has a $\matc^1$-solution $\tilde\rho(t)$ on $[0, t_0)$ with $\tilde\rho(0)>0$. Then for any $f:\matx^n\to\mreals_+$ with
${1\over n}{\Ent(f)\over \EE[f]}\le \tilde\rho(0)$ we have
$$ \Ent(T_t^{\otimes n} f)\le n\tilde \rho(t) \EE[f] \qquad \forall 0 \le t < t_0\,.$$
\end{theorem}

\begin{theorem}\label{th:hca}(Hypercontractivity)  Fix
a non-constant function $f:\matx^n \to \mreals_+$ and $p_0\in(1,\infty)$. Then there is a finite $t_0$ and a unique
function $p(t)$ on $[0,t_0)$ satisfying $\|T_t^{\otimes n} f\|_{p(t)}=\|f\|_{p_0}$. This function is
$\matc^\infty$-smooth, strictly increasing and surjective onto $[p_0,\infty)$ with $p(0)=p_0$. Furthermore, if a
semigroup $T_t$ admits a $(p,\Phi_p)$-LSI for each $p\ge p_0$, then
\begin{align}
	{d\over dt}p(t) &\ge {p(t)(p(t)-1)\over \rho_0} b_{p(t)}\left(p(t) \rho_0\over p(t)-1\right)\,, \qquad \rho_0  =
	{1\over n} \ln {\|f\|_{p_0}\over \|f\|_1}\,.\label{eq:hca_1}
\end{align}
\end{theorem}
\begin{proof}[Proof of Theorem~\ref{th:tenso}] Let us consider the case $n=2$. For a function $f(x_1,x_2)$ denote
by $\Ent_{\pi_i}(f^p)$ the entropy evaluated only along $x_i$, $i=1,2$. Then, from standard chain-rule and convexity of
$\Ent$ we have
\begin{align} \Ent_{\pi_1 \times \pi_2}(f^p) &= \EE_{X_1}[\Ent_{\pi_2}(f^p)] + \Ent_{\pi_1}(\EE_{X_2}[f^p])\\
				  &\le\EE_{X_1}[\Ent_{\pi_2}(f^p)] + \EE_{X_2} [\Ent_{\pi_1}(f^p)] \label{eq:tenso_1}
\end{align}
\apxonly{Note: For any positive $\sigma$-finite measures $\pi_1 \times \pi_2$ we have for all $f\in C_c(\mreals^2)$:
	$$ \Ent_{\pi_1}(\int d\pi_2 f) \le \int d\pi_2 \Ent_{\pi_1}(f)\,.$$
	This follows again by Jensen, since if $f$ is supported on a compact, we can rescale to make $\int d\pi_2$ be
	integral over probability measure. Thus, tensorization works for all general measures too.}

Now, we apply $\Phi_p$-LSI to each term (not forgetting appropriate normalization). For example, for the first term we
get
\begin{align} \EE_{X_1}[\Ent_{\pi_2}(f^p)] &\le
	\EE_{X_1}\left[ \Phi_p\left(\mate_{\pi_2}(f,f^{p-1})\over \EE_{X_2}[f^p] \right) \EE_{X_2}[f^p]
	\right] \\
	&\le
	\Phi_p\left(\EE_{X_1}[\mate_{\pi_2}(f,f^{p-1})]\over \EE_{X_1,X_2}[f^p] \right) \EE_{X_1,X_2}[f^p]\,,
	\label{eq:tenso_2}
\end{align}
where in the second step we used Jensen's inequality and the fact that
$$ (x,y) \mapsto {\Phi\left({x\over y}\right) y} $$
is jointly concave for any concave function $\Phi$. Now plugging~\eqref{eq:tenso_2} (and its analog for the second term)
into~\eqref{eq:tenso_1} and after applying Jensen's inequality again we get
$$ {1\over 2}{\Ent_{\pi_1 \times \pi_2}(f^p)\over  \EE_{X_1,X_2}[f^p] } \le
	\Phi_p \left({1\over 2} {\EE_{X_1}[\mate_{\pi_2}(f,f^{p-1})] + \EE_{X_2}[\mate_{\pi_1}(f,f^{p-1})]\over
	\EE_{X_1,X_2}[f^p] }\right)\,,$$
which is precisely~\eqref{eq:tenso}. The $n>2$ is treated similarly.
\end{proof}

\begin{proof}[Proof of Theorem~\ref{th:mgl}] Since the statement is scale-invariant, we assume $\EE[f]=1$. Define
$\rho(t)\eqdef {1\over n} \Ent(T_t^{\otimes n} f)$. Consider the identity
$$ {d\over dt} \Ent(T_t^{\otimes n} f) = -\mate(T_t^{\otimes n} f, \ln T_t^{\otimes n} f)\,.$$
From tensorizing the $(1,\Phi_1)$-LSI we get
$$ {1\over n} \mate(T_t^{\otimes n} f, \ln T_t^{\otimes n} f) \ge b_1(\rho(t))\,,$$
and hence
$$ \rho'(t) \le -b_1(\rho(t))\,.$$
Let us introduce $\alpha(t) = \ln \rho(t) - \ln \tilde\rho(t)$, then we have for $\alpha(t)$ the following
\begin{equation}\label{eq:mgl_3}
	\alpha'(t) \le -\Psi(\tilde\rho(t) e^{\alpha(t)}) + \Psi(\tilde\rho(t))\,,
\end{equation}
where $\Psi(x) = b_1(x)/x$ is a non-decreasing function of $x\ge 0$. We know $\alpha(0)\le 0$. Suppose that for some
$\tilde t_0>0$
we have $\alpha(\tilde t_0)>0$. Let $t_1=\sup\{0\le t<\tilde t_0: \alpha(t)=0\}$. From continuity of $\alpha$ we have
$t_1<\tilde t_0$,
$\alpha(t_1)=0$ and
$\alpha(t)>0$ for all $t\in(t_1,\tilde t_0]$.  From mean value theorem, we have for
some $t_2\in(t_1,\tilde t_0)$ that $\alpha'(t_2) > 0$. But then from monotonicity of $\Psi$, we have
$$ \Psi(\tilde\rho(t_2) e^{\alpha(t_2)}) - \Psi(\tilde\rho(t_2))\ge 0\,,$$
contradicting~\eqref{eq:mgl_3}. Hence $\alpha(\tilde t_0)\le 0$ for all $\tilde t_0 \in (0,t_0)$.
\end{proof}

\begin{proof}[Proof of Theorem~\ref{th:hca}] The core idea is to integrate the estimates obtained from a non-linear
$p$-LSI. Integrating entropy-energy inequalities have been done before for establishing ultra-contractivity (i.e. for
bounding the kernel function of $T_t$), see e.g.~\cite[Theorem 4.4]{bakry1994hypercontractivite}. However, for $p\to q$
estimates we will need a new idea (see~\eqref{eq:hca_6} below).

Since all the statements are scale-invariant, we assume $\EE[f]=1$. To
avoid clutter, we will write $T_t$ instead of $T_t^{\otimes n}$. We
 define the following function on $\mreals_+^2$
 $$ \phi(t,\xi) \eqdef \ln \|T_t f\|_{1\over \xi}\,.$$
 It is clear that $\phi$ is monotonically decreasing in $\xi$. Steepness of $\phi$ in $\xi$ encodes information about non-uniformity of $T_t f$.
 As time progresses,
 $\xi\mapsto\phi(t,\xi)$ converges to an all-zero function. MGL, LSI and hypercontractivity are estimates on the speed
 of this relaxation.

 We summarize the information we have about $\phi(t,\xi)$ assuming $f$ is non-constant:
 \begin{itemize}
 \item A consequence of H\"older's inequality, cf.~\cite[Theorems 196-197]{HLP88}, implies
 $\xi \mapsto \ln \|g\|_{1\over \xi}$ is strictly convex, unless $g=c 1_S$ (a scaled indicator), in
 which case the function of is linear in $\xi$. Thus, $\phi(t, \xi)$ is convex in $\xi$.
 \item We have
 $$ \phi(0, \xi_2) \ge \phi(0, \xi_1) + \left(\xi_1-\xi_2\right) \ln {1\over \pi^n[\supp f]} \qquad \forall \xi_2 <
 \xi_1$$
 with equality iff $f=c1_S$ (scaled indicator).
 \item Note that $T_t f = 0$ has only $f=0$ as solution (indeed, $\det e^{tL} = e^{\tr L} \neq 0$ since $\matx$ is
 finite). So $\phi(t,\xi)$ is finite and infinitely differentiable in $(t,\xi)$.
 \item The function $t\mapsto \phi(t,\xi)$ is strictly decreasing from $\phi(0,\xi)$ to $0$ for any $\xi<1$ and strictly
 increasing from $\phi(0,\xi)$ to 0 for $\xi>1$. Indeed, $\|T_t f\|_p = \|f\|_p$
 implies $f$ is constant. Furthermore, $\|T_t f\|_{1\over\xi}\to \EE[f]=1$ since $T_t f \to \EE[f]$ as $t\to\infty$.
 \item Consequently, for each $\xi_0$ the fiber
 \begin{equation}\label{eq:hca_7}
 	\{t: \phi(t,\xi_0) = c\}
\end{equation}
 consists of at most one point. Define $t_0$ as the unique solution of
 $$ \phi(t_0, 0) = n\rho_0\,. $$
 Solution exists from continuity of $\phi$ and the fact that $\phi(0,0) = \ln \|f\|_{\infty} > \rho_0 > \phi(+\infty, 0)=0$.
 \item We have the standard identities:
 \begin{align} {\partial \phi\over \partial \xi} &= - {\Ent((T_t f)^{1\over\xi})\over \EE[(T_t f)^{1\over\xi}]}\label{eq:hca_2}\\
    {\partial \phi\over \partial t} &= 
    		- {\mate(T_t f, (T_t f)^{{1\over \xi}-1})\over    \EE[ (T_t f)^{1\over \xi}]}\,.
\end{align}
 \item Since $f$ is non-constant, so is $T_t f$ for all $t\ge 0$ (for otherwise $f-\EE[f]$ is in the kernel of $T_t$).
 Therefore, ${\partial \phi\over \partial \xi} < 0$ for all $(\xi,t)$. Thus, for any $t \in[0,t_0]$ there is at most one
 solution $\xi$ of
 \begin{equation}\label{eq:hca_3}
 	\phi(t, \xi(t)) = \phi\left(0, {1\over p_0}\right) = n\rho_0.
\end{equation}
 $\xi(t)$ is simply a parametrization of the level-set of $\phi$.
 It is clear that $\xi(t)$ is non-increasing. Since fibers~\eqref{eq:hca_7} are singletons, we also conclude that
 $\xi(t)$ is strictly decreasing. From implicit function theorem and ${\partial \phi\over \partial \xi} \neq 0$, we
 infer that solution $\xi(t)$ of~\eqref{eq:hca_3} is $\matc^\infty$-smooth.
 \apxonly{Here I assumed $\xi(0)<1$. Similarly,
 $t\mapsto \xi(t)$ is strictly increasing if $\xi(0)>1$.}
 \item As we mentioned $\xi \mapsto \phi(t, \xi)$ is convex and strictly decreasing. Furthermore, it is strictly convex for $t>0$. From this convexity
 and~\eqref{eq:hca_2} we infer the following important consequences:
 	\begin{align} r &\mapsto {\Ent(f^r)\over \EE[f^r]} \quad \mbox{is increasing in $r\in(0,\infty)$; strictly unless
	$f=c1_S$}\label{eq:hca_4}\\
	   {\Ent(f^r)\over \EE[f^r]} &\ge {\ln \|f\|_{r} - \ln \|f\|_1\over 1-{1\over r}}\,.\qquad r> 1\,.\label{eq:hca_5}
\end{align}	
 \end{itemize}
 We now set $p(t) = {1\over \xi(t)}$, where $\xi(t)$ was found from solving~\eqref{eq:hca_3}.
 From observations after~\eqref{eq:hca_3} we already know that $t\mapsto p(t)$ is well-defined, strictly increasing and
 $\matc^\infty$-smooth. The fact that $p(t)$ is surjective follows from $\xi(t)\to0$ as $t\to t_0$.

 It remains to show~\eqref{eq:hca_1}. This follows from differentiating~\eqref{eq:hca_3}:
 $$ \xi'(t) = -{\mate(t)\over E(t)} \,,$$
 where we defined
 \begin{align} \mate(t) &\eqdef {1\over n} {\mate(T_t f, (T_t f)^{p(t)-1})\over    \EE[ (T_t f)^{p(t)}]}\\
    E(t) &\eqdef {1\over n}{\Ent((T_tf)^{p(t)})\over \EE[(T_tf)^{p(t)}]}\,.\label{eq:hca_9}
\end{align}
 From $(p,\Phi_p)$-LSI we get then
 \begin{equation}\label{eq:hca_8}
 	\xi'(t) \le - {b_{p(t)}(E(t))\over E(t)}\,.
\end{equation}
 Here we arrived at a key new step. Note that without further information about $E(t)$ we can only bound (due
 to convexity of $b_p(\cdot)$) the right-hand side of the above by $-\left. {d\over d s}\right|_{s=0} b_{p(t)}(s)$,
 which would result in a standard, i.e. $\rho_0$-independent, hypercontractivity such as~\eqref{eq:bonami}.
 To improve it, we need to lower-bound $E(t)$ away from $0$.
 Note that from~\eqref{eq:hca_5} we know that $E(0) \ge {\rho_0\over 1-p_0^{-1}}$. To extend this to other times
 we use~\eqref{eq:hca_5} coupled with the fact that $\xi(t)$ is precisely the level-set of $\phi$. Hence, we get
 \begin{equation}\label{eq:hca_6}
 	E(t) \ge {1\over n}{\phi(t,\xi(t))\over 1-\xi(t)} = {\rho_0\over 1-\xi(t)}
\end{equation}
 From convexity of $b_p(\cdot)$, the function $b_p(E)\over E$ is increasing in $E$ and so we can further upper-bound
 $\xi'(t)$ via~\eqref{eq:hca_6} and replacing $\xi(t)$ with $1\over p(t)$ as
 $$ \xi'(t) \le - {b_{p(t)}({p(t)\rho_0\over p(t)-1})\over {p(t)\rho_0\over p(t)-1}}\,.$$
 Noticing that $\xi'(t) = -{p'(t)\over p^2(t)}$ we get~\eqref{eq:hca_1}.
\end{proof}

\apxonly{
\subsection{Alternative estimates and remarks}

\begin{enumerate}
\item Another idea of Alex was to notice that monotonizing $f$ increases $\|T_t f\|_p$ and thus $\Ent(T_t f)$:
	\begin{align} \|T_t f\|_p &\le \|T_t f_{mono}\|_p, \qquad p>1\\
		\|T_t f\|_p &\ge \|T_t f_{mono}\|_p, \qquad p<1\\
	   \Ent(T_t f) &\le \Ent(T_t f_{mono})
	\end{align}
	Then, he asked to what extent we can improve~\eqref{eq:hcx_1} for monotone $f$. The answer is: not much, e.g.
	take $f$ to be the indicator of a subcube of dimension $Rn$, that already gets something like
		$$ \Ent(T_t f) = n(1-R) (\ln 2- h(1-e^{-t}/2)) + nR \rho_{ent}(t) $$
\item Reverse HC: We can get reverse HC with $\xi(0)>1$ and
 $$ \xi'(t) \ge {b_{1\over\xi(t)}(E(t, \xi(t)))\over E(t,\xi(t))}\,,$$
 where $E(t,\xi) = {\Ent((T_tf)^{1\over \xi})\over \EE[(T_t f)^{1\over \xi}]}$. So we are only lacking a good lower
 bound on $E(t,\xi)$.
\end{enumerate}
}

\section{New LSIs and hypercontractivity for the hypercube}\label{sec:hyper}

The fact that we can compare~\cite{Stroock84,Varopoulos85} Dirichlet forms
$\mate(f, f^{p-1})$ with $\mate(f^{p\over 2}, f^{p\over 2})$ immediately leads to the conclusion that for any
reversible semigroup (i.e. $T_t^* = T_t$ in $L_2(\pi)$) we have
\begin{align} b_p(x) &\ge {4|p-1|\over p^2} b_2(x) & \forall x\ge 0\, \forall
p\in(-\infty,\infty)\setminus\{1\}\label{eq:lsi_pto2}\\
   b_p(x) &\ge {1-p\over p^2} b_1(x) & \forall x\ge 0 \, \forall p<1\,.
\end{align}
(see, e.g.,~\cite{bakry1994hypercontractivite} for $p=2$ and~\cite{MOS12} for $p=1$).
Thus, we can get non-trivial non-linear $p$-LSIs by only establishing $p=1,2$ cases (of which $p=2$ was already done
in~\cite{samorodnitsky2008modified}). However, we can also find the sharpest non-linear LSIs for all $p$ explicitly,
which is what we proceed to do.

\begin{theorem}[1-LSI for the hypercube] For all $f:\{0,1\}^n \to (0,\infty)$ with $\EE[f]=1$ we have
\begin{equation}\label{eq:blsi_1}
	{1\over n} \mate(f,\ln f) \ge b_1\left({1\over n}\Ent(f)\right)\,,
\end{equation}
where Dirichlet form is given by~\eqref{eq:diri_cube}, all expectations and $\Ent$ are with respect to uniform
probability measure on $\{0,1\}^n$ and $b_1:[0,\ln 2) \to [0,\infty)$ is a convex increasing function given by
\begin{equation}\label{eq:blsi_2}
	b_1(\ln 2-h(y))= \left(\tfrac{1}{2}-y\right) \ln {1-y\over y}\,, \qquad y\in(0,1/2]\,,
\end{equation}
where $h(\cdot)$ is the binary entropy function~\eqref{eq:hbin_def}.
\end{theorem}
\begin{proof} This result follows from Theorem~\ref{th:lsi_cube} (below) upon taking $p\to 1+$.
\apxonly{By Theorem~\ref{th:tenso}, we only need to work out the case $n=1$. Then, the space of all $f$ with $f(1)
\ge f(0)$ (which is assumed without loss of generality) can be parameterized by $f(0)=2-f(1)=2y$ with $y\in(0,1/2]$. With such a choice we
have $\Ent(f)=\ln 2 - h(y)$ and $\mate(f,\ln f)=({1\over 2}-y) \ln {1-y\over y}$. Thus, we only need to verify that
$b_1$ is increasing and convex. Monotonicity follows from the fact that both factors in the RHS of~\eqref{eq:blsi_2} are
decreasing in $y$. For convexity, we have
\begin{align}
b'(\ln 2 - h(y)) &= 1 - \frac{1-2y}{2y(1-y) \ln \frac{y}{1-y}} \\
b''(\ln 2 - h(y)) &= -\frac{1}{2 \ln \frac{y}{1-y}} \cdot \frac{d}{dy} \frac{1-2y}{y(1-y) \ln \frac{y}{1-y}}
\end{align}
Since $\ln \frac{y}{1-y} < 0$, it suffices to prove $\frac{d}{dy} \frac{1-2y}{y(1-y) \ln \frac{y}{1-y}} \ge 0$. Computing the derivative and rearranging, we need to show $\ln \frac{1-y}{y} \ge \frac{1-2y}{1 - 2y + 2y^2}$.
Substituting $z = \ln \frac{1-y}{y}$ this is the same as
\begin{equation}\label{eq:bpp1}
\ln z \ge \frac{z^2-1}{z^2 + 1}
\end{equation}
for $z \ge 1$.  Both sides of~\eqref{eq:bpp1} vanish at $z = 1$. Comparing the derivatives, we need to verify $1/z \ge
\frac{4z}{\left(z^2 + 1\right)^2}$, which is the same as $\left(z^2 + 1\right)^2 \ge 4z^2$ and hence is true.}
\end{proof}

\begin{corollary} Classical MGL~\eqref{eq:mgl_1} holds.
\end{corollary}
\begin{proof} Since $\Ent(f) \le n\EE[f] \ln 2$, we can define $\tilde\rho(t)=\ln2-m\left(t, {\Ent(f)\over n
\EE[f]}\right)$, where $m(\cdot,\cdot)$ was defined after~\eqref{eq:mgl_1}. A calculation shows that $\tilde\rho(t)$
solves $\tilde\rho(t)' = -b_1(\tilde\rho(t))$ with $b_1$ from~\eqref{eq:blsi_2}. Since $\tilde\rho(0) \ge {1\over
n}{\Ent(f)\over \EE[f]}$, application of Theorem~\ref{th:mgl} completes the proof.
\end{proof}

Next, we proceed to LSI's with $p\neq 1$.
\begin{theorem}[$p$-LSI for the hypercube]\label{th:lsi_cube} Fix $p\in(-\infty,\infty)\setminus\{0,1\}$.
For all $f:\{0,1\}^n \to \mreals_+$  with $\EE[f^p]=1$ (and $f>0$ if $p<1$) we have
\begin{equation}\label{eq:blsi_p}
	{1\over n} \sign(p-1) \mate(f,f^{p-1}) \ge b_p\left({1\over n}\Ent(f^p)\right)\,,
\end{equation}
where the Dirichlet form is given by~\eqref{eq:diri_cube}, all expectations and $\Ent$ are with respect to uniform
probability measure on $\{0,1\}^n$ and $b_p:[0,\ln 2] \to [0,\infty)$ is a convex increasing function given by
\begin{align} b_p(\ln 2 - h(y)) &= {\sign(p-1)\over 2} \left(1 -  y^{{1\over p}} (1-y)^{1-{1\over p}} - y^{1-{1\over
p}} (1-y)^{{1\over p}}\right)\,,
\end{align}
with $0 < y \le {1\over 2}$, and $h(\cdot)$ being the binary entropy function~\eqref{eq:hbin_def}.
\end{theorem}
\begin{remark} Recall that the proof of~\eqref{eq:lsi_gross} in~\cite{LG75} for Ornstein-Uhlenbeck semigroup
was done by first deriving the LSI for the hypercube
and then applying the CLT. Since we derive a better LSI for the hypercube, will we get a better LSI for the
Ornstein-Uhlenbeck? The answer is negative since in the CLT limit we would have $\Ent(f^p)=O(1)$ as $n\to\infty$ and
hence the argument of $b_p$ in~\eqref{eq:blsi_p} converges to $0$ and we get the linear LSI in the limit. In fact, as
noted above,~\eqref{eq:lsi_gross} is tight.
\end{remark}
\begin{proof}By Theorem~\ref{th:tenso}, we only need to work out the case $n=1$. Then, the space of all $f$ can be
parameterized by $f(0)=(2y)^{1\over p}, f(1)=(2-2y)^{1\over p}$ with $y\in[0,1/2]$. Thus we only need verify monotonicity
and convexity.

First, consider the case $p>1$. Let $q={p\over p-1}$. Taking the first derivative, we get
\begin{multline}
b'_p(\ln 2 - h(y)) = -\frac12  \frac{1}{\ln \frac{y}{1-y}} \cdot \bigg(\frac 1p \left(\left(\frac{1-y}{y}\right)^{1/q} -
\left(\frac{y}{1-y}\right)^{1/q}\right) \\{} + \frac 1q \left(\left(\frac{1-y}{y}\right)^{1/p} -
\left(\frac{y}{1-y}\right)^{1/p}\right)\bigg)\label{eq:bplsi_1}
\end{multline}
From here, monotonicity of $b_p$ follows from the fact that the RHS is positive (${1-y\over y} > {y\over 1-y}$). We
proceed to showing convexity.
Let $z = \frac{y}{1-y}$. Then $0 < z \le 1$ and, taking another derivative, we have
\begin{multline}
b''_p(\ln 2 - h(y)) = -\frac12 \frac{1}{(1-y)^2} \cdot  \frac{1}{\ln z} \cdot \frac{d}{dz} \bigg[\frac{1}{\ln z} \cdot
\bigg(\frac 1p \left(\left(\frac{1}{z}\right)^{1/q} - z^{1/q}\right) \\ {} + \frac 1q
\left(\left(\frac{1}{z}\right)^{1/p} - z^{1/p}\right)\bigg)\bigg]
\end{multline}
Since $\ln z < 0$ for $z < 1$, it would suffice to argue that the derivative w.r.t. $z$ on RHS is nonnegative. Let $r(z) = \frac 1p \left(\left(\frac{1}{z}\right)^{1/q} - z^{1/q}\right) + \frac 1q \left(\left(\frac{1}{z}\right)^{1/p} - z^{1/p}\right)$. We need to show $z \ln \frac 1z \cdot \left(-r'\right) \ge r$.

Making another substitution of variables, let $w = \ln z$, that is $-\infty < w \le 0$. Let $t(w) = r(z) = r\left(e^w\right)$. Substituting and simplifying, we need to show $w t'(w) \ge t(w)$.

We have $t(w) = \frac 1p \left(e^{-w/q} - e^{w/q}\right) + \frac 1q \left(e^{-w/p} - e^{w/p}\right)$. Hence
\[
t'(w) = - \frac{1}{pq} \left(e^{-w/q} + e^{w/q} + e^{-w/p} + e^{w/p}\right) \quad \mbox{and}
\]
\[
\quad t''(w) = - \frac{1}{pq} \left(-\frac 1q e^{-w/q} + \frac 1q e^{w/q} -\frac 1p e^{-w/p} + \frac 1p e^{w/p}\right)
\]
In particular, $t$ is a decreasing convex function on $(-\infty,0]$ which vanishes at $0$, and $w t'(w) \ge t(w)$ is satisfied.

Next, consider the case $0 < p < 1$. We repeat the computation above, multiplying throughout by $-1 = \sign(p-1)$. Since
in this case $q < 0$, the sign change cancels out, and the convexity argument, with minor changes as needed, goes
through. For monotonicity, observe that again the signs of both terms in the RHS of~\eqref{eq:bplsi_1} are negative (the
front $-$ sign is canceled by $\sign(p-1)$).

Finally, for the case $p<0$, observe that we can set $g=f^{p-1}$ and apply the already proven inequality to the pair
$(g,{p\over p-1})$ since ${p\over p-1} \in (0,1)$.

\end{proof}

Our chief goal is to derive hypercontractivity inequality tighter than~\eqref{eq:bonami} for functions with small
support. We will replace the constraint on the support size $|\supp f| \le 2^{nR}$ with an analytical proxy:
$$ \|f\|_{p_0} \ge e^{n \rho_0} \|f\|_1\,, \quad \rho_0 = (1-p_0^{-1})(1-R)\ln 2\,,$$
as discussed in~\eqref{eq:hca_p}. We get the following result:

\begin{theorem}\label{th:hcb} Fix $1<p_0 < \infty$ and $0\le \rho_0 \le(1-p_0^{-1})\ln2$. Then the
differential equation
\begin{equation}\label{eq:hcb_0}
	u'(t) = C\left(\rho_0(1+e^{-u(t)})\right)\,, \quad C(\ln 2 - h(y)) = {2-4\sqrt{y(1-y)} \over \ln 2
- h(y)}
\end{equation}
has a unique solution on $[0,\infty)$ with $u(0)=\ln(p_0-1)$. Furthermore, for any $f:\{0,1\}^n \to \mreals_+$ with $\|f\|_{p_0} \ge e^{n\rho_0}
\|f\|_1$ we have
\begin{equation}\label{eq:hcb_1}
	\|T_t^{\otimes n} f\|_{p(t)} \le \|f\|_{p_0}\,, \qquad p(t) = 1+e^{u(t)}\,.
\end{equation}
\end{theorem}
\begin{remark} Ref.~\cite{samorodnitsky2008modified} showed that $C:[0,\ln2]\to[2, 2/\ln2]$ is a smooth,
convex and strictly increasing bijection. Consequently, the function $p(t)$ in~\eqref{eq:hcb_1}  is smooth and satisfies
$$ p(t) > 1+(p_0-1)e^{2t} \qquad \forall t>0 $$
thereby strictly improving the hypercontractivity inequality~\eqref{eq:bonami}. 
Furthermore, it satisfies
\begin{align} p(t) &= p_0 + p'(0) t + {1\over 2} p''(0) t^2 + o(t^2), \qquad t\to0\,, \label{eq:hcb_3}\\
     p'(0) &= (p_0-1) C(x_0)\label{eq:hcb_3a}\\
     p''(0) &= (p_0-1)\left( C(x_0)^2 - C'(x_0) C(x_0) {x_0\over p_0}\right)\\
     x_0 &= {\rho_0 p_0\over p_0-1}\,.\label{eq:hcb_3c}
\end{align}
\apxonly{Using convexity of $C(x)$ we have $C(x) \ge C(x_0) + (x-x_0)C'(x_0)$ and thus we also get the firm bound
$$ p(t) \ge p_0 + p'(0) t + {1\over 2} p''(0) t^2, \qquad \forall t\ge 0\,.$$}
\end{remark}
\begin{remark} Our estimate is locally optimal at $t=0$ in the following sense: for every $q(t)$ such that
$q(0)=p_0$ and $q'(0)>p'(0)$ there exists a function $f$ with $\|f\|_{p_0}\ge e^{n\rho_0} \|f\|_1$ and
$ \|T_t f\|_{q(t)} > \|f\|_{p_0} $
for a sequence of $t\to 0$. This follows from the fact that had a counter-example $q(t)$ existed, it would imply that
the second half of the proof of Theorem~\ref{th:uncert} (see below) could be improved to contradict the first half.
\end{remark}

\begin{proof} First, notice that $C(x) = {4b_2(x)\over x}$, where $b_2$ was defined in Theorem~\ref{th:lsi_cube}. Let
$p_1(t)$ be the function defined by
$$ \|T_t^{\otimes n} f\|_{p_1(t)} = \|f\|_{p_0}\,.$$
Theorem~\ref{th:hca} showed this function to be smooth and growing faster than~\eqref{eq:hca_1}.
From~\eqref{eq:hca_1} and using~\eqref{eq:lsi_pto2} to lower-bound $b_p(\cdot)$ via $b_2(\cdot)$ we get that
$$ p_1'(t) \ge (p_1(t)-1) C\left(p_1(t)\rho_0\over p_1(t)-1\right)\,,$$
or introducing $u_1(t) = \ln(p_1(t) - 1)$ that
$$ u_1'(t) \ge C\left(\rho_0(1+e^{-u_1(t)})\right)\,.$$
The case of $\rho_0 = (1-p_0^{-1})\ln 2$ corresponds to $f$ supported on a single point and can be dealt with
separately. So we assume $\rho_0 < (1-p_0^{-1})\ln 2$, in which case the map
$$ u \mapsto C\left(\rho_0(1+e^{-u})\right) $$
is smooth on some interval $(\ln(p_0-1) - \epsilon, \infty)$. Consequently,~\eqref{eq:hcb_0} possesses a unique solution
with $u(0)=\ln(p_0-1)$ and a Chaplygin-type theorem, e.g.~\cite[Theorem
4.1]{hartman2002ordinary}, implies
$$ u_1(t) \ge u(t)\,. $$
\apxonly{Commented out explicit argument.}
%
%
\end{proof}

For $p_0=2$, we also prove an alternative estimate on $p(t)$ via a method tailored to the hypercube.
\begin{theorem}\label{th:hcc} In the setting of Theorem~\ref{th:hcb} assume $p_0=2$. Then~\eqref{eq:hcb_1} holds with
$p(t)$ given as
\begin{align} p(t) &= 1 + e^{\int_0^t C(\tilde \rho(s)\vee 0)ds}\\
   \tilde \rho(s) &= 2\rho_0 - \ln\left(\frac{2}{1+e^{-2s}}\right) \,.
\end{align}
\end{theorem}
\begin{remark} Using convexity of $C$ we get $C(x\vee 0) \ge C(x_0) + (x-x_0)C'(x_0)$, where $x_0 = \tilde \rho(0) =
2\rho_0$ is from~\eqref{eq:hcb_3c}. Similarly, $\ln {1+e^{-2t}\over
2} \ge -t$. Therefore, altogether we get an explicit estimate:
	\begin{equation}\label{eq:hcc_firm}
		p(t) \ge 1+e^{C(x_0) t - {C'(x_0)\over 2} t^2}\,,
	\end{equation}
	 The $t^2$ term here is, however, worse than that of~\eqref{eq:hcb_3}.
\end{remark}
\begin{proof} We return to~\eqref{eq:hca_8}. Recalling that $\xi(t) = {1\over p(t)}$ and lower-bounding $b_p$ by $b_2$
via~\eqref{eq:lsi_pto2} we get
$$ {d\over dt}\ln(p(t)-1) \ge C(E(t))\,,$$
(with $E(t)$ from~\eqref{eq:hca_9}), which implies after integration
$$ p(t) \ge 1 + e^{\int_0^s C(E(s))ds}\,.$$
Since $E(t)\ge 0$ by definition it only suffices to prove \begin{equation}
E(t) \ge 2\rho_0 - \ln\left(\frac{2}{1+e^{-2t}}\right) \label{eq:hcc_2}
\end{equation}

Next, we obtain a lower bound on $\|T_t^{\otimes n} f\|_2$. To that end introduce a function $\Lambda^{\otimes n}_t$ on $\{0,1\}^n$ with the property $T_t^{\otimes n} f = \Lambda^{\otimes
n}_t \ast f$. Note that $\Lambda^{\otimes n}_t(x) = \left(1-e^{-t}\right)^{|x|} \left(1+e^{-t}\right)^{n-|x|}$, where
$|x|$ denotes the Hamming weight of $x$. Clearly $\Lambda^{\otimes n}_t \ge 0$. Since $f\ge 0$ we have
\[
\|T_t^{\otimes n} f\|_2^2 = \left<T_t^{\otimes n} f, T_t^{\otimes n} f\right> = \left<\Lambda^{\otimes n}_t \ast f, \Lambda^{\otimes n}_t \ast f\right> =
\]
\[
\left<\Lambda^{\otimes n}_t \ast \Lambda^{\otimes n}_t, f \ast f\right> = \left<\Lambda^{\otimes n}_{2t}, f \ast f\right> \ge \frac{1}{2^n} \Lambda^{\otimes n}_{2t}(0) \cdot (f\ast f)(0) = \left(\frac{1 + e^{-2t}}{2}\right)^n \cdot \|f\|^2_2
\]

To prove~\eqref{eq:hcc_2}, observe that by~\eqref{eq:hca_4}~and~\eqref{eq:hca_5} and by the preceding calculation,
\[
E(t) \ge \frac 1n \ln \frac{\|T_t^{\otimes n} f\|_2^2}{\|T_t^{\otimes n} f\|_1^2} = \frac 1n \ln \frac{\|T_t^{\otimes n}
f\|_2^2}{\|f\|_1^2} \ge \frac 1n \ln \frac{\|f\|_2^2}{\|f\|_1^2} - \ln\left(\frac{2}{1+e^{-2t}}\right) \ge \tilde\rho(t)
\]

We remark that the main difference in this proof compared to Theorem~\ref{th:hcb} is in using a different idea for
lower-bounding the entropy $E(t)$. Theorem~\ref{th:hcb} essentially relied on~\eqref{eq:hca_6}.
\end{proof}

\apxonly{\subsection{Questions about exact asymptotics}
Some remarks about exact asymptotics. First, consider the following problem
$$ \sup_{f: \|f\|_{p_0} \ge e^{n\rho_0}, \EE[f]=1} \|T_t^{\otimes n} f\|_{p}  \qquad \rho_0 \le (1-p_0^{-1}) \ln 2,
p\ge p_0$$
One might hope that it tensorizes (i.e. solution is iid). But unfortunatelly, comparing iid $f$ with (scaled)
indicator of a ball gives:
\begin{enumerate}
\item for $p$ close to $p_0$ ball achieves larger spreading $\|T_t f\|_p$.
\item for $p$ far from $p_0$, iid solution is better.
\item In all cases I saw the two curves (fixed $t$, varying $p$)  had only one point of intersection.
\item Surprisingly, this intersection seems to always occur before $\|T_t f\|_p$ reaches the value of $e^{n\rho_0}$.
\end{enumerate}
The last observation, lead to the following conjecture.

Given $(\matx, T_t, \pi)$ and its tensorization, let us introduce the following asymptotic problem: Fix $p_0>1$ and
$0 < \rho_0 < (1-p_0^{-1})\ln 2$. Define
$$ p(t, p_0, \rho_0) \eqdef \sup\{p:  \|T_t^{\otimes n} f\|_p
\le \|f\|_{p_0} \quad\forall n\ge 1 \, \forall f: \EE[f]=1, \|f\|_{p_0}\ge e^{n\rho_0} \} $$

Then we have the following lower bound (this is single-letter, i.e. computed for $n=1$!):
$$ p_{lb}(t, p_0, \rho_0) \eqdef \sup\{p: \min_{1 \le q\le p_0} \ln \|T_t\|_{q\to p} + \rho_0 {1-{1\over q}\over 1-{1\over p_0}} \le \rho_0 \} $$
(This follows from $\|T_t^{\otimes n}\|_{q\to p} = \|T_t\|_{q\to p}^n$ and $\|f\|_q \le \|f\|_{p_0}^{1-{1\over q}\over 1-{1\over p_0}}$.)

On the other hand, we have an upper bound $p_{ub}(t,p_0, \rho_0)$, obtained from solving  (with exponential precision)
$$ \|T_t 1_{B_\alpha n}\|_{p_{ub}} = e^{n(\rho_0 - h(\alpha) + \ln 2)} $$
where $\alpha$ is itself determined from solving
$$ \|1_{B_\alpha n}\|_{p_0} = e^{n(\rho_0 - h(\alpha) + \ln 2)} $$

\textbf{Conjecture:} (somewhat analytically and somewhat numerically verified)
$$ p(t,p_0, \rho_0) = p_{lb} = p_{ub} $$
In other words, \underline{balls asymptotically are the worst} for support-constrained hypercontractivity.

}

\section{Uncertainty principle on the hypercube}\label{sec:up}

\subsection{Background}

Uncertainty principle asserts that a function and its Fourier transform cannot be simultaneously narrowly concentrated.
There are several approaches to quantifying this statement, and here we adopt the Hilbert space point of view,
cf.~\cite[Chapter 3]{havin1994uncertainty}. Namely, for a pair of subspaces $V_1,V_2$ of a Hilbert space with inner
product $(\cdot,\cdot)$ and $\|f\|_2^2 \eqdef (f,f)$ we define
	$$ \cos \angle(V_1,V_2) \eqdef \sup_{f_1\in V_1, f_2 \in V_2} {|(f_1, f_2)|\over \|f_1\|_2 \|f_2\|_2}\,.$$
For the uncertainty principle, we will select sets $S$ and $\Sigma$ and define subspaces
 \begin{align}
 	V_S &\eqdef \{f: \supp f \subset S\}\label{eq:def_vs}\\
 	\hat V_\Sigma &\eqdef \{f: \supp \hat f \subset \Sigma\}\,, \label{eq:def_vhs}
\end{align}	
	where $\hat f$ denotes the corresponding Fourier transform (we will define it
precisely). Uncertainty principle corresponds to bounding $\cos \angle(V_S,\hat V_\Sigma)$ away from 1, thus
establishing to what extent functions can simultaneously concentrate on $(S,\Sigma)$.

There is a number of equivalent ways to think of $\cos \angle(V_1,V_2)$. Letting $P_i$ be an orthogonal
projection on $V_i$ and $P_i^\perp$ projection on $V_i^\perp$, it can be shown~\cite[Chapter 3]{havin1994uncertainty}:
\begin{align}
	\cos \angle(V_1, V_2) \le \theta & \iff \forall f \in V_1: \quad \|P_2 f\|_2 \le \theta \|f\|_2\label{eq:pdr_1}\\
				  & \iff \forall f \in V_1: \quad \|P_2^\perp f\|_2 \ge \sqrt{1-\theta^2} \|f\|_2\\
				  & \iff \lambda_{\max}(P_1 P_2 P_1) \le \theta\\
				  & \iff \|P_1 P_2\|_{2\to 2} \le \sqrt{\theta}\\
				  & \iff \forall f: \|f\|_2^2 \le {1\over 1-\theta} \left( \|P_1^\perp f\|^2_2 + \|P_2^\perp
f\|^2_2\right)\label{eq:pdr_4}
\end{align}
\apxonly{For the last statement: just take $e_i={P_i f\over \|P_i f\|}$ and represent $f=c_1 e_1+c_2 e_2 + f_0$, where
$f_0 \perp (e_1,e_2)$.}
Furthermore, there is also a simple criterion:
$$ \cos \angle(V_1, V_2) <1 \iff V_1 \cap V_2 = \{0\} \text{~and~} (V_1+V_2) \mbox{~---~closed}\,,$$
where for finite-dimensional $V_i$'s the closedness condition is vacuous (but not in general).

Finally, as shown in~\cite{fuchs1954magnitude} and~\cite{landau1961prolate}, knowledge of
$\cos\angle(V_1,V_2)$ is sufficient for completely characterizing the two-dimensional region
$$ \{(\|P_1f\|_2^2, \|P_2 f\|_2^2) \} $$

Before proceeding to our own results, we briefly review the history of results for
$\mreals^n$. First,~\cite{slepian1961prolate} computed $\cos \angle(V_S, \hat V_\Sigma)$
for $S,\Sigma$ being two balls (in fact
they computed $\lambda_{max}(P_1 P_2 P_1)$ and named eigenfunctions of the latter prolate spheroidal functions).
Next,~\cite{benedicks1985fourier} (worked out in 1974, but published much later) showed that
$$ \vol(S), \vol(\Sigma) <\infty \quad \implies \quad V_S \cap \hat V_\Sigma = \{0\}\,.$$
Later,~\cite{amrein1977support} strengthened this to
$$ \vol(S), \vol(\Sigma) <\infty \quad \implies \quad \cos\angle(V_S,V_\Sigma) <1\,.$$
Finally, for $n=1$ ~\cite{nazarov1993local} showed
$$ \vol(S), \vol(\Sigma) <\infty \quad \implies \quad \cos\angle(V_S,V_\Sigma) <1-ce^{-c\vol(S)\vol(\Sigma)}\,.$$
Lately, there were a number of extensions and improvements of this result for $n>1$, e.g.~\cite{jaming2007nazarov}.

\subsection{Sharp uncertainty principle on $\FF_2^n$}
Define the characters, indexed by $v\in\FF_2^n$,
$$ \chi_v(x) \eqdef \prod_{j: v_j = 1} \chi_j(x) = (-1)^{\la v,x \ra}\,,$$
where $\la v,x \ra=\sum_{j=1}^n v_j x_j$ is a non-degenerate bilinear form on $\FF_2^n$. The
Fourier transform of $f:\FF_2^n \to \mcomplex$ is
$$ \hat f(\omega) \eqdef \sum_{x\in\FF_2^n} \chi_{\omega}(x) f(x) = 2^n (f, \chi_\omega)\,, \qquad
\omega \in \FF_2^n\,.$$
We denote by $|x|$ the Hamming weight of $x\in \FF_2^n$ and by $B_r = \{x: |x| \le r\}$ -- Hamming ball.

\begin{theorem}\label{th:uncert} For any $\rho_1,\rho_2 \in [0,1/2]$ satisfying
\begin{equation}\label{eq:bc1}
	(1-2\rho_1)^2 + (1-2\rho_2)^2 > 1\,,
\end{equation}
there exist an $\epsilon>0$ and $n_0$ such that for any $n\ge n_0$, any $S\subset \FF_2^n$ with $|S|\le e^{nh(\rho_1)}$ and
$\Sigma=B_{\rho_2 n}$ we have
\begin{equation}\label{eq:uncert}
	\cos \angle (V_S, \hat V_{\Sigma}) \le e^{-n \epsilon}\,,
\end{equation}
where $V_S, \hat V_\Sigma$ are defined in~\eqref{eq:def_vs}-\eqref{eq:def_vhs}.

Conversely, for $\rho_1,\rho_2 \in [0,1/2]$ satisfying\footnote{When $\rho_1>{1\over2}$ (or $\rho_2> {1\over 2}$),
the result~\eqref{eq:uncert_anti} also holds by reducing to $\rho_1={1\over2}$. This is possible since $\cos\angle(V_S, \hat
V_{\Sigma})$ is monotone in $S,\Sigma$.}
\begin{equation}\label{eq:bc2}
	(1-2\rho_1)^2 + (1-2\rho_2)^2 < 1\,,
\end{equation}
there exist $\epsilon>0$ and $n_0$ such that for all $n\ge n_0$ we have
\begin{equation}\label{eq:uncert_anti}
	\cos \angle(V_S, \hat V_{\Sigma}) \ge 1-e^{-n\epsilon}, \qquad S=B_{\rho_1 n}, \Sigma=B_{\rho_2 n}\,.
\end{equation}
\end{theorem}

\begin{proof}
For the case~\eqref{eq:bc2}, fix $\alpha \in \mreals$, $\alpha \neq \pm 1$ and consider the following Fourier pair:
\begin{align} f(x) &= \alpha^{|x|} \\
   \hat f(\omega) &= c \left(1-\alpha \over 1+\alpha\right)^{|\omega|}, \qquad c = (1+\alpha)^n\,.
\end{align}
Then, it is easy to see that the $L_2$-norm of $f$ is concentrated around $|x| \approx {1\over 1+\alpha^{-2}} n$.
Thus, whenever radius
$ \rho_1> {1\over 1+\alpha^{-2}}$, we have for some $\epsilon>0$
$$ \sum_{x: |x| > \rho_1 n} f(x)^2 \le e^{-\epsilon n} \sum_{x\in\FF_2^n} f(x)^2\,. $$
Similarly, whenever $\rho_2 > {1\over 1+\beta^{-2}}$, where $\beta = {1-\alpha\over 1+\alpha}$, we have
$$ \sum_{\omega: |\omega| > \rho_2 n} \hat f(\omega)^2 \le e^{-\epsilon n} \sum_{\omega \in \FF_2^n} \hat f(\omega)^2\,.$$
Whenever,~\eqref{eq:bc2} holds, it is not hard to see that there exists a choice of $\alpha \in (0,1)$ satisfying both
$ \rho_1> {1\over 1+\alpha^{-2}}$ and $\rho_2 > {1\over 1+\beta^{-2}}$. Thus, taking corresponding $f$ and
using~\eqref{eq:pdr_4} we get~\eqref{eq:uncert_anti}.

Next, we assume~\eqref{eq:bc1}. We define the Fourier projection operators $\Pi_a$ as
\begin{equation}\label{eq:fproj}
	 \widehat{\Pi_a f}(\omega) \eqdef \hat f(\omega) 1\{|\omega| = a\}\,,\qquad a=0,1,\ldots,n\,.
\end{equation}
and set $\Pi_{\le r} = \sum_{a=0}^r \Pi_a$. By~\eqref{eq:pdr_1} we need to show that for any function $f$ with support
$|\supp f|\le e^{nh(\rho_1)}$ we have
$$ \|\Pi_{\le \rho_2 n} f\|_2 \le e^{-n\epsilon} \|f\|_2\,,$$
for some $\epsilon>0$ independent of $n$ and $f$.

Note that $\widehat{T_t f}(\omega) = e^{-t|\omega|} \hat f(\omega)$. Thus, comparing eigenvalues we have $e^{ta}
T_t\succeq \Pi_a$ (in the sense of positive-semidefiniteness). Consequently,
\begin{equation}\label{eq:bc3}
	\|\Pi_a f\|_2^2 = (\Pi_a f, f) \le e^{at} (T_t f, f) \le e^{at} \|f\|_{q} \|T_t f\|_p\,,
\end{equation}
where $p$ and $q$ are H\"older conjugates. Since $|\supp f|\le e^{nh(\rho_1)}$ we have
from Theorem~\ref{th:hcb} with $p_0=2$ and $\rho_0 = {\ln2-h(\rho_1)\over2}$:
$$ \|T_t f\|_{p(t)} \le \|f\|_2\,, \qquad p(t) = 2 + p'(0) t + o(t)\,,$$
where the value of $p'(0)$ is given in~\eqref{eq:hcb_3a}.\footnote{For extracting explicit constants, one may
invoke~\eqref{eq:hcc_firm} instead.}
Taking $p=p(t)>2$ in~\eqref{eq:bc3} we need upper-bound $\|f\|_{q}$, which we again do by invoking the bound on support
\begin{equation}\label{eq:hcc_3}
	\|f\|_{q} \le \|f\|_2 e^{-n(\ln 2- h(\rho_1))({1\over q} - {1\over2})}\,,\qquad \forall 1 \le q \le 2\,.
\end{equation}
Overall, we have shown for all $a$ and $t$ that
$$ \|\Pi_a f\|_2^2 \le e^{at} e^{-n(\ln2 - h(\rho_1))({1\over 2} - {1\over p(t)})} \|f\|_2^2\,.$$
Analyzing this inequality for $t$ close to $0$ we conclude that whenever
\begin{equation}\label{eq:bc9}
	\rho_2 < {p'(0)\over 4}(\ln 2-h(\rho_1))
\end{equation}
we necessarily have for some $\epsilon>0$ (depending on the gap in the inequality above and on the local bound for $p''(t)$ at
0) that for all $a\le \rho_2 n$
$$ \|\Pi_a f\|_2 \le e^{-n\epsilon} \|f\|_2\,.$$
Using expression for $p'(0)$ in~\eqref{eq:hcb_3a}, we see that~\eqref{eq:bc9} is equivalent to
\begin{equation}\label{eq:bc10}
	2\rho_2 < 1-2\sqrt{\rho_1 (1-\rho_1)}\,,
\end{equation}
which is in turn equivalent to~\eqref{eq:bc1}.

\apxonly{Here is a different method. It gives a slightly better bound, but requires to use a less-suitable HC $|T_t f\|_2
\le \|f\|_p$.
Note that
\begin{equation}\label{eq:bc7}
	\|\Pi_a f\|_2 \le e^{s a} \|T_s f\|_2\,.
\end{equation}
Indeed, comparing eigenvalues we have $e^{ta} T_t \succeq \Pi_a $ (in the sense of positive-semidefiniteness) and hence
\begin{equation}\label{eq:bc8}
	\|\Pi_a f\|_2^2 = \la \Pi_a f, f \ra \le e^{ta} \la T_t f, f \ra = e^{ta} \la T_{t/2} f, T_{t/2} f \ra = e^{ta} \|T_{t/2} f\|_2^2\,,
\end{equation}
where we used self-adjointedness of $T_{t/2}$ and the semigroup property $T_t = T_{t/2} T_{t/2}$. Thus,~\eqref{eq:bc7}
follows from~\eqref{eq:bc8} by taking $s=t/2$. Since $|\supp f|\le 2^{nh(\rho_1)}$ we have
from Theorem~\ref{th:hcb} with $R=h(\rho_1)$
$$ \|T_s f\|_2 \le \|f\|_{p_0} $$
etc.

Vanilla HC gives the following:
By the Kahn-Kalai-Linial bound, e.g.~\cite[Lemma 11]{YP13}, we have
$$ \|\Pi_a f\|_2 \le (p-1)^{-a/2}\|f\|_p\,, \qquad \forall p\in(1,2) $$
and use~\eqref{eq:hcc_3} to get that whenever
$$ {a\over n} < {\log 2\over 2} (1-h(\rho_1)) + o(1)\,,$$
then
\begin{equation}\label{eq:ltt0}
	\|\Pi_a f\|_2 \le e^{-n\epsilon} \|f\|_2\,.
\end{equation}
Note: it is not hard to convince oneself that taking $p=2-\epsilon$ does not lose generality here (i.e. optimizing over
$p$ does not extend the range of $a$ for which~\eqref{eq:ltt0} is shown).
In other words, we have $\cos\angle \le e^{-n\epsilon}$ if
$$ \rho_2 < {\log 2\over 2} (1-h(\rho_1))\,. $$
}
\end{proof}

For completeness, we also provide a criterion for when two subspaces have a common element (for the special case of
$S,\Sigma$ being two balls). It demonstrates that there is a ``discontinuity'' between the regime $\cos \angle (V_S, \hat V_\Sigma) \ge 1-e^{O(n)}$ and $\cos \angle (V_S, \hat V_\Sigma) = 1$.
\begin{proposition} Let $S = B_{r_1}$ and $\Sigma = B_{r_2}$ in $\FF_2^n$. Then
	$$ V_S \cap \hat V_{\Sigma} \neq \{0\} \iff \cos \angle (V_S, \hat V_\Sigma) = 1 \iff r_1 + r_2 \ge n\,.$$
\end{proposition}
\begin{proof} If $r_1 + r_2 \ge n$, then take $f(x) = 1\{x_{r_1+1}=\cdots =x_n=0\}$. Its Fourier
transform is supported on $\{\omega: \omega_1=\cdots=\omega_{r_1}=0\}$. Thus $f\in V_S \cap \hat V_{\Sigma}$. On the
other hand, suppose there is $f\in V_S \cap \hat V_{\Sigma}$. By averaging over permutations of coordinates (both
subspaces are invariant to such), we conclude that $f(x) = f_1(|x|)$. As such, it can be expanded in terms of Krawtchouk
polynomials:
$$ f_1(|x|) = \sum_{k=0}^n a_k K_k(|x|)\,,$$
where each $K_k(\cdot)$ is a degree $k$ univariate polynomial. Note that $\hat K_k(\omega) \neq 0$ iff $|\omega|=k$.
Thus, constraint $\supp \hat f \subset B_{r_2}$ is equivalent
to requiring $a_{k}=0$ for $k > r_2$. Thus, we conclude that $f_1$ on integers inside $[0,n]$ coincides with a degree
$r_2$ polynomial, and hence has $\le r_2$ zeros. Thus, $r_1 \ge n-r_2$ as claimed.
\end{proof}

\subsection{Discussion}

To start the discussion, let us recall the function
\begin{equation}\label{eq:rlp1_def}
	R_{LP1}(\delta)\eqdef
h\left({1\over 2} - \sqrt{\delta(1-\delta)}\right)\,,
\end{equation}
which is known as the first linear-programming (LP1) bound~\cite{MRRW77}. Its importance is in that it gives an upper bound
$2^{n R_{LP1}(\delta) + o(n)}$ on the number of points in Hamming space $\{0,1\}^n$ that have pairwise distance
exceeding $n\delta$. In the range $\delta \gtrsim 0.28$ this bound is the best known to date, whereas for smaller
$\delta$ it is superceded by the second linear-programming bound~\cite{MRRW77}.

It is instructive, next, to provide an equivalent statement of Theorem~\ref{th:uncert}.
\begin{theorem}[Restatement of the uncertainty principle]\label{th:uncert_rest} For any $\delta < 1/2$ and $0<E<R_{LP1}(\delta)$ there is
$\epsilon>0$ with the following property. Let $f(x_1,\ldots,x_n)$ be polynomial of total degree at most $\delta
n$. Then, for any $S\subset \{\pm1\}^n$ of size $|S| \le e^{n E}$ we have
	$$ \sum_{x\in S} f(x)^2 \le e^{-n\epsilon} \sum_{x\in\{\pm1\}^n} f(x)^2\,.$$
\end{theorem}

First, we mention that a weaker estimate with $E=(1-3\delta)\ln 2$ was shown by~\cite{KM_uncert} by
using hypercontractivity~\eqref{eq:bonami} similarly to~\cite{KKL88}. Their argument can be easily tightened to yield a
stronger $E=\ln 2 - 2\delta$. Both of these are weaker than the sharp estimate above.

Theorem~\ref{th:uncert_rest} shows that any low-degree polynomial (restricted to the hypercube) smears its $L_2$-norm so
evenly that one needs to sum $e^{n R_{LP1}(\delta)}$ top values in order to obtain a sizable fraction of its overall
$L_2$-norm. It is interesting to compare this with~\cite{AS_Delsarte} showing that any $f$ that is a) a degree $\le \delta
n$ polynomial and b) $f\ge 0$ satisfies
\begin{equation}\label{eq:up_l1}
	\max_{x\in \{\pm1\}^n} |f(x)| \le e^{-n(\ln 2 - h({\delta\over2})) + o(n)} \sum_{x\in\{\pm1\}^n} |f(x)|\,.
\end{equation}
We \textit{conjecture} that~\eqref{eq:up_l1} holds for all $f:\{\pm1\}^n\to\mreals$ of degree $\le \delta n$. This could
be called an $L_1$-version of the uncertainty principle. If true, it would imply that the sum of any $e^{n
R_{Ham}(\delta)}$, $R_{Ham}(\delta)=\ln 2 - h({\delta\over2})$ values of $|f(x)|$ is negligible compared to the sum over
all of $\{\pm 1\}^n$.

\apxonly{Notes: \begin{itemize}
	\item To prove this, just notice that $\Theta_L = \min\{2^n {\EE[\hat g] \over \hat g(0)}: \hat g\ge 0, g=0
	\text{~on~} [\delta n+1,n]\cap\ZZ\}$. Thus $\hat g(0) \le $ (blah) and by translating $\hat g_1(x) \eqdef \hat
	g(x-x_0)$, we also have that $\hat g(x_0) \le $ (blah).
	\item From~\eqref{eq:hca_2}, this implies $\Ent(f)\le nh(\delta/2) + o(n)$.
	\item Alex also proves (Lemma 3.2
ibid) that $f(0) \le e^{-nh(\delta/2)}\|f\|_1$ for any permutation-invariant (but not nec. non-negative) $f$.
	\end{itemize}
}

Finally, we discuss to what extent one can relax condition that $\Sigma=B_{\rho_2 n}$ in Theorem~\ref{th:uncert}. First,
notice that clearly the same conclusion holds for $\Sigma$ which is an image of a ball $B_{\rho_2 n}$ under a linear
isomorphism $\FF_2^n \to \FF_2^n$. This provides a wealth of examples of $\Sigma$ that are less ``contiguous'' than
$B_{\rho_2 n}$. \apxonly{Indeed, such isomorphism also acts as isomorphism in the primal space.}

At the same time, we cannot extend Theorem~\ref{th:uncert} to $\Sigma$ being an arbitrary subset of the same cardinality
as $B_{\rho_2 n}$ (recall that $|B_{\rho_2 n}| = e^{nh(\rho_2) + o(n)}$). Indeed, a simple computation shows that when $S$ and $\Sigma$ are linear subspaces of $\FF_2^n$ we
have
\begin{equation}\label{eq:cos_lin}
	\cos \angle(V_S, \hat V_\Sigma) = \sqrt{|\Sigma \cap S^\perp|\over |S^\perp|}\,,
\end{equation}
where $S^\perp \eqdef \{x: <x,v> = 0\,, \forall v \in S\}$ is the dual of $S$. \apxonly{Taking $G:\FF_2^k \to \FF_2^n$ with
image $S$, we define $f_1(u) = f(uG)$ and notice that $\sum_{\omega} \hat f(\omega)^2 = |\Sigma \cap S^\perp| \sum_{\nu
\in \Sigma'} \hat f_1(\nu)^2$, where $\Sigma' = \{\nu: \nu = \omega G', \omega \in \Sigma\}$.}
Thus, if we take $S$ to be a linear subspace of dimension $\alpha n$, $0<\alpha<1$, and $\Sigma = S^\perp$ (of dimension $(1-\alpha)n$)
and solve for $\rho_1$ and $\rho_2$ in
	$$ h(\rho_1) = \alpha \ln 2, h(\rho_2) = (1-\alpha) \ln 2 $$
we conclude that these $\rho_1$ and $\rho_2$ always satisfy
	$$ (1-2\rho_1)^2 + (1-2\rho_2)^2 > 1\,,$$
while from~\eqref{eq:cos_lin} we have $\cos\angle (V_S,\hat V_\Sigma)=1$.

Consequently, we leave open the question of
determining the more general uncertainty principle, i.e. characterizing the best
pairs $(E_1, E_2)$ for which one can prove implication
$$ |S|\le e^{n E_1}, |\Sigma|\le e^{n E_2} \quad \implies \quad \cos \angle(V_S, \hat V_\Sigma) \le \epsilon\,.$$
A partial result easily follows from the Hausdorff-Young inequality:

\begin{proposition}\label{prop:setset} For any $E_1,E_2 \in (0,\ln 2)$ satisfying $E_1 + E_2 < \ln 2$ there exist $\epsilon>0$ and $n_0$
such that for all $n\ge n_0$, all $S,\Sigma\subset \FF_2^n$ with $|S|=e^{n E_1}, |\Sigma|=e^{n E_2}$ we have
	$$ \cos \angle (V_S, \hat V_\Sigma) \le 1-\epsilon\,.$$
Conversely, for any positive integers $k_1,k_2 \le n$ such that $k_1 + k_2 \ge n$ there exist $|S|=2^{k_1}$ and $|\Sigma|=2^{k_2}$ such that
	$$ \cos \angle (V_S, \hat V_\Sigma) = 1\,.$$
\end{proposition}
\begin{proof} Second part follows from~\eqref{eq:cos_lin}. For the first part, let $\theta = \cos \angle(V_S, \hat
V_\Sigma)$ and $E_1 + E_2 = \ln 2 - \delta$ for $\delta>0$. We will show that
\begin{equation}\label{eq:ccl_0}
	\theta^2 \le 1- {\delta - {1\over n}\ln 2\over \ln 2 - \max(E_1,E_2)}\,.
\end{equation}
Without loss of generality, suppose $E_2\ge E_1$.
Recall that a simple consequence of the Hausdorff-Young inequality is the Hirschmann (or entropic) uncertainty
principle~\cite[Exercise 4.2.10]{tao2006additive}: For any $f:\FF_2^n\to\mreals$ we have
\begin{equation}\label{eq:ccl_1}
		{\Ent(f^2)\over \EE[f^2]} + {\Ent(\hat f^2)\over \EE[\hat f^2]} \le \ln |\FF_2^n| = n\ln 2\,.
\end{equation}	
Thus, taking $f$ supported on $S$ we estimate (from Jensen's inequality)
\begin{equation}\label{eq:ccl_2}
		{\Ent(f^2)\over \EE[f^2]} \ge n \ln 2 - \ln |S| = n(\ln 2 - E_1)\,.
\end{equation}	
Suppose that ${\EE[\hat f^2 1_\Sigma]\over \EE[\hat f^2]} = \theta^2$, and introduce a random variable $U$ taking values
in $\FF_2^n$ with
	$$ \PP[U=u] \eqdef {\hat f^2(u) \over \sum_\omega\hat f^2(\omega)}\,.$$
	Then, we have ${\Ent(\hat f^2)\over \EE[\hat f^2]} = n \ln 2 - H(U)$, with $H(\cdot)$ denoting the Shannon entropy.
	Introducing also $T=1\{U\in \Sigma\}$ we get by the chain rule
	\begin{align} n \ln2 - {\Ent(\hat f^2)\over \EE[\hat f^2]} &=
		H(U) = H(U,T) = H(T) + H(U|T) \\
			&\le \ln 2 + \theta^2 \ln |\Sigma| + (1-\theta^2) \ln |\Sigma^c| \\
			&\le \ln 2 + n(\theta^2 E_2 + (1-\theta^2) \ln 2)\,. \label{eq:ccl_3}
	\end{align}			
Altogether, from~\eqref{eq:ccl_1},~\eqref{eq:ccl_2} and~\eqref{eq:ccl_3} we get~\eqref{eq:ccl_0}.
\end{proof}

\subsection{A similar result for Euclidean space}

It is interesting to observe that a result analogous to Theorem~\ref{th:uncert} in $\mreals^n$ with Lebesgue measure follows from the
sharp form of Young's inequality~\cite{WB75}. This provokes us to hypothesize that the refined
hypercontractivity result on the hypercube (Theorem~\ref{th:hcb}) could play the role of the sharp Young inequality (or
Babenko-Beckner inequality~\cite{Babenko61,WB75}) in $\mreals^n$.

Notation: In this section we define $B_r=\{x\in\mreals^n: \|x\| \le r\}$, $\|x\|^2=(x,x)$, $(x,y)=\sum_{k=1}^n x_k y_k$, $|S|$ -- the
Lebesgue measure of $S$, $\|f\|_p = \left(\int_{\mreals^n} |f(x)|^p dx\right)^{1\over p}$ and for $f\in L_1\cap L_2$
	$$ \hat f(\omega) =  \int_{\mreals^n} e^{-2\pi i(\omega,x)} f(x) dx\,, \qquad \omega \in \mreals^n\,,$$
with the standard extension by continuity to all of $f\in L_2$.
\begin{theorem}\label{th:uncert_eucl} For any $\rho_1,\rho_2 > 0$ satisfying
\begin{equation}\label{eq:bce1}
	\rho_1 \rho_2 < {1\over 4\pi}
\end{equation}
there exist an $\epsilon>0$ and $n_0$ such that for any $n\ge n_0$, any $S\subset \mreals^n$ with $|S|=|B_{\rho_1
\sqrt{n}}|$ and $\Sigma=B_{\rho_2 \sqrt{n}}$ we have
\begin{equation}\label{eq:uncert_eucl}
	\cos \angle (V_S, \hat V_{\Sigma}) \le e^{-n \epsilon}\,,
\end{equation}
where $V_S, \hat V_\Sigma$ are defined in~\eqref{eq:def_vs}-\eqref{eq:def_vhs}.

Conversely, for $\rho_1,\rho_2 \ge 0$ satisfying
\begin{equation}\label{eq:bce2}
	\rho_1 \rho_2 > {1\over 4\pi}
\end{equation}
there exist $\epsilon>0$ and $n_0$ such that for all $n\ge n_0$ we have
\begin{equation}\label{eq:uncert_eucl_anti}
	\cos \angle(V_S, \hat V_{\Sigma}) \ge 1-e^{-n\epsilon}, \qquad S=B_{\rho_1 \sqrt{n}}, \Sigma=B_{\rho_2 \sqrt{n}}\,.
\end{equation}
\end{theorem}
\begin{remark}Recall that a standard Heisenberg-Weyl uncertainty (in dimension 1) states that for all $f:\mreals \to \mreals$ with $\int f^2 = 1$
we have
	$$ \left(\int x^2 f^2(x) \right) \left(\int \omega^2 \hat f^2(\omega)\right) \ge {1\over 16\pi^2}\,.$$
	So the product of mean-square widths of $f$ and $\hat f$ should exceed ${1\over 4\pi}$, in accord with our
	estimate.
	\end{remark}
\begin{proof} Since the statement is asymptotic, we will use the standard fact
$$ \ln |B_1| =  {n\over 2} \ln{2\pi e \over n} - {1\over 2}\ln (\pi n) + O({1\over n}) $$
and thus
\begin{equation}\label{eq:bce_0}
	\ln |B_{\rho \sqrt{n}}| = {n\over 2} \ln(2\pi e \rho^2) + O(\ln n)\,.
\end{equation}
To prove the second part, consider the Fourier pair (for any $\sigma>0$):
\begin{align} f(x) &= {1\over (2\pi \sigma^2)^{n\over2}} e^{-{\|x\|^2\over 2\sigma^2}}\\
   \hat f(\omega) &= e^{-2\pi^2 \sigma^2 \|\omega\|^2}\,.
\end{align}
Choose $\sigma>0$ so that $\rho_1 > {\sigma \over \sqrt{2}}$ and $\rho_2 > {1\over
\sqrt{8} \sigma \pi}$ (which is possible due to~\eqref{eq:bce2}).
From concentration of Gaussian measure, it is easy to check that for some $\epsilon>0$ we have
\begin{align} \|f 1_{B_{\rho_1 \sqrt{n}}}\|_2 &\ge (1-e^{-n\epsilon}) \|f\|_2\\
   \|\hat f 1_{B_{\rho_2 \sqrt{n}}}\|_2 &\ge (1-e^{-n\epsilon}) \|\hat f\|_2
\end{align}
and therefore~\eqref{eq:uncert_eucl_anti} follows from~\eqref{eq:pdr_4}.

For the first part, recall a sharp form of the Young inequality on $\mreals^n$ (from~\cite{WB75})
	\begin{equation}\label{eq:sharp_young}
		\|f*g\|_r \le \left({C_p C_q\over C_r}\right)^n \|f\|_p \|g\|_q\,, \qquad C_{s} =
				e^{{1\over 2} ({\ln s\over s}  + {s-1\over s} \ln (1-s^{-1}) )}
	\end{equation}	
	valid for $1\le p,q,r\le \infty$ and ${1\over p} + {1\over q} = 1 + {1\over r}$. Consider the heat semigroup
	$$ e^{t\Delta} f \eqdef f*\phi_t, \qquad \phi_t(x) = {1\over (4\pi t)^{n\over2}} e^{-{\|x\|^2\over 4t}}\,. $$
	Let $\gamma>0$ be a constant to be specified later, and for a real $t<{1\over 2\gamma}$ we set ${1\over
	p(t)}={1\over 2} - \gamma t$. Then apply~\eqref{eq:sharp_young} with $r=p(t), p=2$ and $q=q(t)$ given by
	${1\over q(t)}= 1-\gamma t$ to get, after some calculations, a hypercontractive inequality
\begin{equation}\label{eq:bce_1}
		\|e^{t\Delta} f\|_{p(t)} \le e^{n E(t)} \|f\|_2\,,
\end{equation}	
	where
	$$ E(t) = {\gamma t\over 2} \ln {\gamma \over \pi e^2} + o(t)\,, \qquad t\to 0\,.$$
	Now, we proceed as in the proof of Theorem~\ref{th:uncert} with~\eqref{eq:bce_1} replacing the use of the
	more precise hypercontractivity for the cube.

	Namely, we define the ball-multiplier operator
	$$ \widehat{\Pi_r f}(\omega) = \hat f(\omega) 1_{B_r}(\omega)\,. $$
	Now consider a function $f$ supported on $S$ and note the chain
	\begin{align} \|\Pi_r f\|_2^2 = \la \Pi_r f, f \ra &\le e^{4\pi^2 r^2t} \la e^{t\Delta} f,f \ra \label{eq:bce_2} \\
					  &\le e^{4\pi^2 r^2t} \|f\|_{q(t)} \|e^{t\Delta} f\|_{p(t)}\label{eq:bce_3}\\
					  &\le e^{4\pi^2 r^2t + nE(t)} \|f\|_2 \|f\|_{\tilde q(t)}\label{eq:bce_4}\\
					  &\le e^{4\pi^2 r^2t + nE(t) + \gamma t \ln |S|} \|f\|_2^2 \label{eq:bce_5}\,,
	\end{align}
	where in~\eqref{eq:bce_2} we used the fact that $\widehat{e^{t\Delta} f}(\omega)=e^{-4\pi^2 \|\omega\|^2} \hat
	f(\omega)$, in~\eqref{eq:bce_3} we used H\"older's inequality with $\tilde q(t)$ denoting the conjugate of
	$p(t)$,~\eqref{eq:bce_4} is by~\eqref{eq:bce_1}, and~\eqref{eq:bce_5} is by invoking the bound on the support of
	$S$ via H\"older's inequality
	$$ \|f\|_{\tilde q} \le \|f\|_2 |S|^{{1\over \tilde q} - {1\over 2}}\,.$$
	Taking $r=\rho_2\sqrt{n}$ and using~\eqref{eq:bce_0} to estimate $|S|$, we conclude that
	\begin{equation}\label{eq:bce_6}
		\|\Pi_{r} f\|_2 \le e^{-n \epsilon} \|f\|_2\,,
\end{equation}	
	whenever there is a $\gamma>0$ such that
	$$ 4\pi^2 \rho_2^2 + {\gamma \over 2} \ln {2\rho_1^2 \gamma\over e} < 0\,.$$
	Since $\min_{\gamma >0} \gamma \ln(a\gamma) = -{1\over ea}$, we get that~\eqref{eq:bce_6} holds whenever
	$$ 4\pi^2 \rho_2^2 - {1\over 4\rho_1^2} < 0\,, $$
	which is equivalent to~\eqref{eq:bce1}.
\end{proof}

The structure of the proof for $\mreals^n$ suggests that perhaps it is worthwhile to look for a general inequality on
the hypercube that could replace the use of hypercontractivity in the proof of Theorem~\ref{th:uncert}, i.e. play a role
similar to that of the sharp Young inequality on $\mreals^n$ (of course, the Young
inequality itself cannot be sharpened on the hypercube, or on any finite group).
\apxonly{\textbf{TODO:} I am also somewhat certain, that Theorem~\ref{th:uncert} should be provable from the analysis of
$\|T_t\|_{p\to q}$ outside of hypercontractivity region.}

To complete the parallel with the hypercube, we also note that Proposition~\ref{prop:setset} (uncertainty principle for
general supports) also has an $\mreals^n$-analog.

\begin{proposition} For any $S,\Sigma\subset \mreals^n$ with $|S|=|B_{\rho_1 \sqrt{n}}|$ and
$|\Sigma|=|B_{\rho_2 \sqrt{n}}|$ with
\begin{equation}\label{eq:eucl_setset}
		\rho_1 \rho_2 < {1\over 2\pi e}
\end{equation}	
we have $ \cos(\angle V_S, \hat V_{\Sigma}) \le e^{-n \epsilon}$.
\end{proposition}
\begin{proof}
Denoting by $P_2$
the operator of orthogonal projection on $\hat V_\Sigma$ and taking $f\in V_S$ we get
$$ \|P_2 f\|_2 = \left(\int_\Sigma |\hat f(\omega)|^2 d\omega\right)^{1\over 2} \le \|\hat f\|_\infty |\Sigma|^{1\over2}
	\le \|f \|_1 |\Sigma|^{1\over2} \le
\|f\|_2 (|S|\cdot|\Sigma|)^{1\over 2}$$
and the rest follows from~\eqref{eq:bce_0} and~\eqref{eq:pdr_1}.
\end{proof}

We do not think~\eqref{eq:eucl_setset} is sharp. In fact, it is natural to conjecture that the sharp
constant in~\eqref{eq:eucl_setset} should be ${1\over 4\pi}$, that is that the pair of balls present the worst case for
the uncertainty principle. For the latter, see also the discussion in~\cite[Section I]{jaming2007nazarov}.

\apxonly{

\subsection{Open questions related to uncertainty principle}

List:
\begin{enumerate}
\item Sharp uncertainty bound would follow if the iid estimate
$$ \|\Pi_{\delta n}\|_{2\leftarrow 2-\epsilon} \ge
\sup_{f_1}{\|\Pi f_1^{\otimes n}\|_2\over \|f_1\|_{2-\epsilon}^n} $$
was exponentially tight. So this gives hope that maybe it is.

\end{enumerate}

}

\section{Application: lower bound on spectrum of sparse Boolean functions}\label{sec:ft_coding}

In this section we will use base-2 binary entropy defined as
$$ h_2(\rho) = -\rho \log_2(\rho) - (1-\rho)\log_2(1-\rho)\,,$$
and denote $h_2^{-1}:[0,1] \mapsto [0,1/2]$ its functional inverse.

Consider a sparse Boolean function $\Psi:\FF_2^k \to \{0,1\}$ with $|\supp \Psi|=n$. It is clear that every Fourier coefficient of $\Psi$ satisfies:
$$ \hat\Psi(\omega) = \sum_{x \in \FF_2^k} (-1)^{\langle \omega, x\rangle} \Psi(x) \in [-n, n]\,.$$
What we show below is that Fourier coefficients $\hat \Psi(\omega)$ are large ($\Theta(n)$) even for large frequencies
$\omega$, i.e. $|\omega|\ge c k$, where $|\cdot|$ denotes Hamming weight.

\begin{theorem} Fix $\tau, \rho_1', \rho_2'>0$ such that
	$$ (1-2\rho_1')^2 + (1-2\rho_2')^2 > 1 $$
	There exists $\delta_k \to 0$ such that for $k$, and every Boolean function $\Psi:\FF_2^k \to \{0,1\}$ with $|\supp
	\Psi|=n$, $n \in [\tau k, k/\tau]$ we have
	$$ \max_{\omega: |\omega| \ge \rho_2' k} \hat\Psi(\omega) \ge n (2\sqrt{\rho_1 (1-\rho_1)} + \delta_k)\,,$$
	where $\rho_1 = h_2^{-1}(h_2(\rho_1') \tfrac{k}{n})$.
\end{theorem}
\begin{proof} It will turn out to be more convenient to prove this estimate in the language of linear maps, which we
will do in the next section. Here we notice how to convert to that statement.
Given $\Psi$ define operator $A$ via $Ah = \Psi \ast h$, with $\ast$ denoting convolution. Define also
numbers $d_r$ via
	$$ n-2d_r = \max_{\omega: |\omega|\ge r} \hat \Psi(\omega)\,.$$
	(see~\eqref{eq:ltt1} for an equivalent definition). Then the proof of Theorem~\ref{th:map}, or more
	exactly~\eqref{eq:ltt3}, shows the stated bound.
\end{proof}

\apxonly{\textbf{Question:} Does it hold for non-Boolean functions? Not without lower-bounding the smallest non-zero
element (obviously). Also support cannot be replaced with (e.g.) $|\Psi|_{\ell_1} = n$, since we can take $\Psi =
n2^{-n} 1$.}

\subsection{Restatement as a property of linear maps (coding theory)}

We now restate the previous result as a curious property of linear maps between binary spaces.

\begin{theorem}
\label{th:map}
For any $0<R'<R<1$ there exists $\delta_n\to0$ such that for any linear map $f:\FF_2^k\to\FF_2^n$ with
${k\over n} = R$ there exists an $x\in\FF_2^k$ s.t.
\begin{align}{1\over n}|f(x)| &\le \delta_{LP1}(R') + \delta_n \label{eq:ltt_a}\\
	{1\over k}|x| &\ge  \delta_{LP1}\left({R'\over R}\right) - \delta_n\,, \label{eq:ltt_b}
\end{align}
where $\delta_{LP1}( h_2(\rho)) = {1\over 2} - \sqrt{\rho(1-\rho)}$ is the inverse of the earlier $R_{LP1}(\delta)$
function in~\eqref{eq:rlp1_def}.
\end{theorem}
\begin{remark} This estimate significantly outperforms previously best known bounds of  this kind~\cite[Theorem
1]{YP15-abmaps}, but only applies to linear maps.
\end{remark}
We give two different proofs, in two subsections below. Note that the two proofs take slightly different points of view.
The first proof deals with linear maps $f:\FF_2^k\to\FF_2^n$, while the second proof looks rather at images of these
maps, linear codes in $\FF_2^n$. In particular, in the second proof we assume that the image of $f$ is of dimension $k$
(i.e. $f$ is of full rank).

\subsection{Method 1 -- graph covers}\label{sec:map1}

\begin{proof}
To every linear map $f:\FF_2^k\to\FF_2^n$ we associate the following increasing sequence of numbers:
$$ d_{r}(f) \eqdef \min\{|f(x)|: |x| \ge r\}\,,$$
where $d_1$ is just the minimum distance of $f$. Note that, as in coding theory, we think of elements of $\FF_2^k$ and
$\FF_2^n$ as row-vectors and thus map $f$ can be represented as a binary $k \times n$ matrix, whose columns we denote by
$c_1,\ldots, c_n$. Following~\cite{friedman2005generalized} we also
associate to $f$ a Cayley graph $\Gamma$ with vertices $\FF_2^k$ and generators $\{c_i, i=1,\ldots, n\}$. (We will use freely facts from~\cite{friedman2005generalized}, perhaps in a somewhat different formulation, from now on.) Then
\begin{equation}\label{eq:ltt1}
	n-2d_r = \max\left\{ {(Ah,h)\over \|h\|_2^2}: \hat h = 0 \mbox{~on ball $B(0,r-1)$}\right\}\,,
\end{equation}
where $A$ is the adjacency matrix of $\Gamma$. Note that $A$ is also a convolution operator on $\FF_2^k$:
\begin{equation}\label{eq:adef}
	Ah = h*\left(\sum_{i=1}^n \delta_{c_i}\right)\,.
\end{equation}
As in~\cite{friedman2005generalized}, select a covering map $\FF_2^n\to\FF_2^k$ and take $B\subset\FF_2^n$ to be the Hamming ball of radius
$n\rho_1$, with $0<\rho_1<{1\over2}$ found as $h_2(\rho_1)=R'$. There exists a function $g_B:\FF_2^n \to \mreals$, supported on $B$ with the
property:
$$ (A_C ~g_B, g_B) \ge \lambda_B \|g_B\|_2^2\,,$$
where $A_C$ is the adjacency matrix of the $n$-dimensional hypercube, and $\lambda_B = 2n\sqrt{\rho_1(1-\rho_1)} + o(n)$.

Hence, there exists a function $h_B:\FF_2^k \to \mreals$ supported on the image of $B$  under the covering map  with the property:
$$ (A h_B, h_B) \ge \lambda_B \|h_B\|_2^2\,,$$
and $|\supp h_B| \le |B| = 2^{nh_2(\rho_1) + o(n)}$.
Then to get a lower bound on~\eqref{eq:ltt1} we set
$$ h = h_B - \Pi_{<r} h_B\,,$$
where $\Pi_{<r}=\sum_{a<r} \Pi_a$ and $\Pi_a$ is from~\eqref{eq:fproj}.

Note that $A$ and $\Pi_{<r}$ commute and eigenvalues of $A$ are bounded by $n$, so $(A\Pi_{<r} h_B,h_B)=(A\Pi_{<r}
h_B,\Pi_{<r} h_B) \le n (\Pi h_B,h_B)$. We then have:
$$ (Ah,h) = (Ah_B, h_B) - (A\Pi_{<r} h_B, h_B) \ge \lambda_B \|h_B\|_2^2 - n (\Pi_{<r}h_B, h_B) $$
Thus,
\begin{equation}\label{eq:ltt3}
	n-2d_{r} \ge 2n\sqrt{\rho_1(1-\rho_1)} + o(n)
\end{equation}
whenever
$$ \|\Pi_{<r} h_B\|^2_2 \le \|h_B\|_2^2\cdot o(1)\,. $$
Using the uncertainty principle for the $k$-dimensional cube (Theorem~\ref{th:uncert}) we estimate
$$ \|\Pi_{<r} h_B\| \ll \|h_B\|^2\,,$$
as long as 
\begin{equation}\label{eq:ltt2}
		{r\over k}< {1\over 2} - \sqrt{\rho_1' (1-\rho_1')}\,,
\end{equation}
where $\rho_1'$ is found from $h_2(\rho_1') = {h_2(\rho_1)\over R} = {R'\over R}$. After simple algebra, we see
that~\eqref{eq:ltt3}-\eqref{eq:ltt2} are equivalent to~\eqref{eq:ltt_a}-\eqref{eq:ltt_b}.

\end{proof}

\medskip
\subsection{Method 2 -- analytic}


We start an uncertainty-type claim for subspaces of $\FF_2^n$.

Let $C$ be a $k$-dimensional linear subspace $C$ of $\FF_2^n$. Given a basis $\mathbf{v} = \Big\{v_1,...,v_k\Big\}$ of $C$, denote the length of representation of a vector $x \in C$ in terms of $V$ by $|x|_{\mathbf{v}}$.

\begin{lemma}
\label{lem:subspace-uncertainty}
Let $f$ be a function supported on a subset $A \subseteq \FF_2^n$. Let $0 \le r \le k \le n$ be integer parameters such
that $k \ge \log_2 |A|$, and, moreover, writing $|A| = 2^{h_2(\rho_1) \cdot k}$, ${k \choose r} = 2^{h_2(\rho_2) \cdot k}$, we have $(1 - 2 \rho_1)^2 + (1 - 2 \rho_2)^2 ~>~ 1$.

Then, for any $k$-dimensional subspace $C$ of $\FF_2^n$ and for any basis $\mathbf{v}$ of $C$ holds
\[
\sum_{\omega \in C, |\omega|_{\mathbf{v}} \le r} \widehat{f}^2(\omega) \ll \sum_{\omega \in C} \widehat{f}^2(\omega)
\]
Here the $\ll$ sign means that the LHS is exponentially smaller than the RHS.
\end{lemma}

\begin{proof}
Let $F = f \ast 1_{C^{\perp}}$.

Note that $F$ is constant on cosets of $C^{\perp}$ and that $\widehat{F}(\omega) = \left\{\begin{array}{ccc} |C^\perp| \widehat{f}(\omega) & \mbox{if} & \omega \in C \\ 0 & & \mbox{otherwise} \end{array} \right.$.

Let $M$ be a $k \times n$ matrix with rows $v_1,...,v_k$. We define a function $g$ on $\FF_2^k$ as follows. For $x \in \FF_2^k$, the pre-image $\{y \in \FF_2^n, M y = x\}$ is a coset of $C^{\perp}$, and we set $g(x)$ to be the (fixed) value of $F$ on this coset.

Next, we calculate the Fourier transform of $g$. Let $\alpha \in \FF_2^k$. Let $\omega = \alpha^t M \in C$. We
claim that $\widehat{g}(\alpha) = \widehat{f}(\omega)$. To see this, note that for any $y$ such that $My = x$ holds
$\la x,\alpha \ra = \la My, \alpha \ra = \la y, M^t \alpha \ra = \la y,\omega \ra$. Using this we compute
\[
\widehat{g}(\alpha) = \sum_{x \in \FF_2^k} g(x) (-1)^{\la x,\alpha\ra} = {1\over |C^\perp|}\sum_{x \in \FF_2^k} \sum_{y:
My = x} F(y) (-1)^{\la y,\omega\ra} = {1\over |C^\perp|}\widehat{F}(\omega) = \widehat{f}(\omega)
\]

Next, we apply the uncertainty principle for $g$ on $\FF_2^k$. Observe that the cardinality of the support of $g$ is given by the number of cosets of $C^{\perp}$ intersecting $A$, which is at most $|A|$. The constraints on $|A|$, $k$, and $r$ imply $\sum_{|\alpha| \le r} \widehat{g}^2(\alpha) \ll \sum_\alpha \widehat{g}^2(\alpha)$, which is equivalent to the claim of the lemma.

\end{proof}

We now prove Theorem~\ref{th:map}, first restating it for linear codes rather than for linear maps.
\begin{theorem}
\label{thm:code}
Let $0 < R < 1$. Let $C \subseteq \FF_2^n$ be a linear code of rate $k = Rn$. Let $\mathbf{v} = \Big\{v_1,...,v_{k}\Big\}$ be a basis of $C$. Then for any $0 \le R' < R$ there is a vector $x \in C$ with
\[
\frac1n |x| \le \delta_{LP1}(R') + \delta_n  \quad \text{and} \quad \frac1k |x|_{\mathbf{v}} \ge  \delta_{LP1}\left(\frac{R'}{R}\right) - \delta_n
\]
\end{theorem}

\begin{proof}

Let $r = h_2^{-1}(R') \cdot n$. Let $B$ be the Hamming ball of radius $r$ around zero in $\FF_2^n$. As in \cite{friedman2005generalized}, let $g_B$ be a function supported on $B$, with the property:
$$ (A_C ~g_B, g_B) \ge \lambda_B \|g_B\|_2^2\,,$$
where $A_C$ is the adjacency matrix of the $n$-dimensional hypercube, and $\lambda_B = 2n\sqrt{\frac{r}{n}(1-\frac{r}{n})} + o(n)$.

Let $d = \frac{n - \lambda_B + 1}{2}$. Note that $d = \left(\frac12 - \sqrt{h_2^{-1}(R')\left(1-h_2^{-1}(R')\right)}\right) \cdot n  + o(n) = \delta_{LP1}(R') \cdot n + o(n)$.

Note that $|B| = 2^{R' \cdot n} \le |C| = 2^k$. We introduce two additional parameters with a view towards using Lemma~\ref{lem:subspace-uncertainty}. Let
$\rho_1$ be such that $|B| = 2^{h_2(\rho_1) \cdot k}$, and let $\rho_2$ satisfy $(1 - 2 \rho_1)^2 + (1 - 2 \rho_2)^2 = 1$. Computing explicitly,
\[
\rho_1 = h_2^{-1}\left(\frac{R'}{R}\right) \quad \text{and} \quad \rho_2 = \delta_{LP1}\left(\frac{R'}{R}\right)
\]

We proceed with the following computation, as in~\cite{navon2009linear}. Let $F = |C| \cdot g_B \ast 1_{C^{\perp}}$. Compute $(A_C F, F)$ in two ways.
On one hand, since $A_C$ commutes with convolutions, we have $(A_C ~F, F) \ge \lambda_B \cdot (F,F) = \frac{\lambda_B}{2^n} \sum_x \widehat{F}^2(x)$.
On the other hand, observe that $A_C F = F \ast \left(\sum_{i=1}^n \delta_{e_i}\right)$, where $e_i$ is the $i^{\tiny{th}}$ unit vector (compare with~\eqref{eq:adef}). Hence $\widehat{A_C F}(x) = \widehat{F}(x) \cdot \left(\sum_{i=1}^n \widehat{\delta_{e_i}}(x)\right) = \widehat{F}(x) \cdot \left(\sum_{i=1}^n (-1)^{x_i}\right) =
(n-2|x|) \cdot \widehat{F}(x)$. And therefore $(A_C ~F, F) = \frac{1}{2^n} \sum_x (n - 2|x|) \cdot \widehat{F}^2(x)$. Substituting $\lambda_B = n - 2d + 1$, we get the inequality $\sum_x (n - 2|x|) \cdot \widehat{F}^2(x) \ge (n-2d+1) \cdot \sum_x \widehat{F}^2(x)$. Rearranging and simplifying, this implies that $2d \cdot \sum_{|x| \le d} \widehat{F}^2(x) \ge \sum_x \widehat{F}^2(x)$. 

Since $\widehat{F} = |C| \cdot \widehat{g_B} \cdot \widehat{1_{C^{\perp}}} = 2^n \cdot \widehat{g_B} \cdot 1_C$, we deduce that 
\[
2d \cdot \sum_{x \in C,|x| \le d} \widehat{g_B}^2(x) \ge \sum_{x \in C} \widehat{g_B}^2(x).
\]

We now apply Lemma~\ref{lem:subspace-uncertainty} for $g_B$. By the lemma, we can choose a sequence $\delta_n \rightarrow 0$, so that for $r < \left(\rho_2 - \delta_n\right) \cdot k$ holds 
\[
\sum_{x \in C, |x|_{\mathbf{v}} > r} \widehat{g_B}^2(x) > \left(1 - \frac{1}{2d}\right) \cdot \sum_{x \in C} \widehat{g_B}^2(x).
\]

Combining these two inequalities, we deduce that $\sum_{x \in C, |x| \le d, |x|_{\mathbf{v}} > r} \widehat{g_B}^2(x) > 0$, implying that there exists a vector $x \in C$ such that $|x| \le d$ and $|x|_{\mathbf{v}} > r$, proving the claim of the theorem.

\end{proof}

\section*{Acknowledgement}

The work of Y.P. was supported (in part) by the National Science Foundation under Grant No CCF-13-18620, and by the
Center for Science of Information (CSoI), an NSF Science and Technology Center, under grant agreement CCF-09-39370.
The work of A. S. was supported (in part) by grants from the US-Israel Binational Science Foundation and from the Israel Science Foundation.


\begin{thebibliography}{10}

\bibitem{amrein1977support}
W.~Amrein and A.~Berthier.
\newblock On support properties of $l_p$-functions and their {F}ourier
  transforms.
\newblock {\em J. Func. Anal.}, 24(3):258--267, 1977.

\bibitem{Babenko61}
K.~Babenko.
\newblock An inequality in the theory of {F}ourier integrals.
\newblock {\em Izv. Akad. Nauk SSSR, Ser. Mat}, 25:531--542, 1961.

\bibitem{bakry1994hypercontractivite}
D.~Bakry.
\newblock L'hypercontractivit{\'e} et son utilisation en th{\'e}orie des
  semigroupes.
\newblock In {\em Lectures on probability theory}, pages 1--114. Springer,
  1994.

\bibitem{bakry2006functional}
D.~Bakry.
\newblock Functional inequalities for {M}arkov semigroups.
\newblock In {\em Probability measures on groups}, pages 91--147. Tata
  Institute of Fundamental Research, Mubai, 2006.

\bibitem{WB75}
W.~Beckner.
\newblock Inequalities in {F}ourier analysis.
\newblock {\em Ann. Math.}, 102(1):159--182, July 1975.

\bibitem{beckner1995pitt}
W.~Beckner.
\newblock Pitt's inequality and the uncertainty principle.
\newblock {\em Proc. Amer. Math. Soc.}, 123(6):1897--1905, 1995.

\bibitem{benedicks1985fourier}
M.~Benedicks.
\newblock On {F}ourier transforms of functions supported on sets of finite
  {L}ebesgue measure.
\newblock {\em J. Math. Anal. Appl.}, 106(1):180--183, 1985.

\bibitem{bobkov1998modified}
S.~G. Bobkov and M.~Ledoux.
\newblock On modified logarithmic {S}obolev inequalities for {B}ernoulli and
  {P}oisson measures.
\newblock {\em Journal of functional analysis}, 156(2):347--365, 1998.

\bibitem{Bobkov2006}
S.~G. Bobkov and P.~Tetali.
\newblock Modified logarithmic {S}obolev inequalities in discrete settings.
\newblock {\em Journal of Theoretical Probability}, 19(2):289--336, Jun 2006.

\bibitem{AB70}
A.~Bonami.
\newblock {\'E}tude des coefficients de {F}ourier des fonctions de $l_p(g)$.
\newblock {\em Ann. Inst. Fourier (Grenoble)}, 20(2):335--402, 1970.

\bibitem{Carlen91}
E.~A. Carlen.
\newblock Superadditivity of {F}isher's information and logarithmic {S}obolev
  inequalities.
\newblock {\em Journal of Functional Analysis}, 101(1):194 -- 211, 1991.

\bibitem{costa1984similarity}
M.~Costa and T.~Cover.
\newblock On the similarity of the entropy power inequality and the
  {B}runn-{M}inkowski inequality (corresp.).
\newblock {\em {IEEE} Trans. Inf. Theory}, 30(6):837--839, 1984.

\bibitem{davies1984ultracontractivity}
E.~B. Davies and B.~Simon.
\newblock Ultracontractivity and the heat kernel for {S}chr{\"o}dinger
  operators and {D}irichlet {L}aplacians.
\newblock {\em J. Func. Anal.}, 59(2):335--395, 1984.

\bibitem{DSC96}
P.~Diaconis and L.~Saloff-Coste.
\newblock Logarithmic {S}obolev inequalities for finite {M}arkov chains.
\newblock {\em Ann. Appl. Probab.}, 6(3):695--750, 1996.

\bibitem{erbar2012ricci}
M.~Erbar and J.~Maas.
\newblock Ricci curvature of finite {M}arkov chains via convexity of the
  entropy.
\newblock {\em Archive for Rational Mechanics and Analysis}, 206(3):997--1038,
  2012.

\bibitem{friedman2005generalized}
J.~Friedman and J.-P. Tillich.
\newblock Generalized {A}lon--{B}oppana theorems and error-correcting codes.
\newblock {\em SIAM Journal on Discrete Mathematics}, 19(3):700--718, 2005.

\bibitem{fuchs1954magnitude}
W.~H.~J. Fuchs.
\newblock On the magnitude of {F}ourier transforms.
\newblock In {\em Proc. Int. Math. Cong., Amsterdam}, pages 106--107, 1954.

\bibitem{LG75}
L.~Gross.
\newblock Logarithmic {S}obolev inequalities.
\newblock {\em Amer. J. Math.}, 97:1061--1083, 1975.

\bibitem{HLP88}
G.~H. Hardy, J.~E. Littlewood, and G.~Polya.
\newblock {\em Inequalities}.
\newblock Cambridge University Press, 1988.

\bibitem{hartman2002ordinary}
P.~Hartman.
\newblock {\em Ordinary differential equation}.
\newblock John Wiley \& Sons, New York, USA, 1964.

\bibitem{havin1994uncertainty}
V.~Havin and B.~J{\"o}ricke.
\newblock {\em The uncertainty principle in harmonic analysis}.
\newblock Springer, 1994.

\bibitem{jaming2007nazarov}
P.~Jaming.
\newblock Nazarov's uncertainty principles in higher dimension.
\newblock {\em J. Approximation Th.}, 149(1):30--41, 2007.

\bibitem{KKL88}
J.~Kahn, G.~Kalai, and N.~Linial.
\newblock The influence of variables on {B}oolean functions.
\newblock In {\em Proc. 29th Ann. Symp. on Foundations of Comp. Sci.}, pages
  68--80, Los Alamitos, CA, 1988.

\bibitem{KM_uncert}
J.~Kahn and R.~Meshulam.
\newblock Uncertainty inequalities on {H}amming cubes.
\newblock manuscript.

\bibitem{landau1961prolate}
H.~J. Landau and H.~O. Pollak.
\newblock Prolate spheroidal wave functions, {F}ourier analysis and
  uncertainty-- {II}.
\newblock {\em Bell Syst. Tech. J.}, 40(1):65--84, 1961.

\bibitem{MRRW77}
R.~McEliece, E.~Rodemich, H.~Rumsey, and L.~Welch.
\newblock New upper bounds on the rate of a code via the
  {D}elsarte-{M}ac{W}illiams inequalities.
\newblock {\em {IEEE} Trans. Inf. Theory}, 23(2):157--166, 1977.

\bibitem{miclo1999majoration}
L.~Miclo.
\newblock Une majoration sous-exponentielle pour la convergence de l'entropie
  des cha{\^\i}nes de {M}arkov {\`a} trou spectral.
\newblock {\em Ann. Inst. H. Poincar{\'e} Probab. Statist}, 35(3):261--311,
  1999.

\bibitem{MOS12}
E.~Mossel, K.~Oleszkiewicz, and A.~Sen.
\newblock On reverse hypercontractivity.
\newblock {\em Geometric and Functional Analysis}, 23(3):1062--1097, 2013.

\bibitem{navon2009linear}
M.~Navon and A.~Samorodnitsky.
\newblock Linear programming bounds for codes via a covering argument.
\newblock {\em Discrete \& Computational Geometry}, 41(2):199--207, 2009.

\bibitem{nazarov1993local}
F.~L. Nazarov.
\newblock Local estimates for exponential polynomials and their applications to
  inequalities of the uncertainty principle type.
\newblock {\em Algebra i Analiz (in Russian)}, 5(4):3--66, 1993.

\bibitem{EN66}
E.~Nelson.
\newblock A quartic interaction in two dimensions.
\newblock In R.~Goodman and I.~Segal, editors, {\em Mathematical Theory of
  Elementary Particles}, Cambridge, MA, 1966. M.I.T. Press.

\bibitem{YP15-abmaps}
Y.~Polyanskiy.
\newblock On metric properties of maps between {H}amming spaces and related
  graph homomorphisms.
\newblock {\em J. Combin. Theory Ser. A}, 145:227--251, 2017.

\bibitem{saloff1997lectures}
L.~Saloff-Coste.
\newblock Lectures on finite {M}arkov chains.
\newblock In {\em Lectures on probability theory and statistics}, pages
  301--413. Springer, 1997.

\bibitem{AS_Delsarte}
A.~Samorodnitsky.
\newblock Extremal properties of solutions for {D}elsarte's linear program.
\newblock {\em preprint}.

\bibitem{samorodnitsky2008modified}
A.~Samorodnitsky.
\newblock A modified logarithmic {S}obolev inequality for the {H}amming cube
  and some applications.
\newblock {\em arXiv preprint arXiv:0807.1679}, 2008.

\bibitem{slepian1961prolate}
D.~Slepian and H.~O. Pollak.
\newblock Prolate spheroidal wave functions, {F}ourier analysis and uncertainty
  -- {I}.
\newblock {\em Bell Syst. Tech. J.}, 40(1):43--63, 1961.

\bibitem{stam1959some}
A.~Stam.
\newblock Some inequalities satisfied by the quantities of information of
  {F}isher and {S}hannon.
\newblock {\em Inf. Contr.}, 2(2):101--112, 1959.

\bibitem{Stroock84}
D.~W. Stroock.
\newblock {\em An introduction to the theory of large deviations}.
\newblock Universitext, Springer-Verlag, New York, 1984.

\bibitem{tao2006additive}
T.~Tao and V.~H. Vu.
\newblock {\em Additive combinatorics}, volume 105.
\newblock Cambridge University Press, 2006.

\bibitem{Varopoulos85}
N.~T. Varopoulos.
\newblock {H}ardy-{L}ittlewood theory for semigroups.
\newblock {\em J. Functional Analysis}, 63(2):240--260, 1985.

\bibitem{weissler1978logarithmic}
F.~B. Weissler.
\newblock Logarithmic {S}obolev inequalities for the heat-diffusion semigroup.
\newblock {\em Trans. Amer. Math. Soc.}, 237:255--269, 1978.

\bibitem{witsenhausen1974entropy}
H.~Witsenhausen.
\newblock Entropy inequalities for discrete channels.
\newblock {\em {IEEE} Trans. Inf. Theory}, 20(5):610--616, 1974.

\bibitem{wyner1973theorem}
A.~D. Wyner and J.~Ziv.
\newblock A theorem on the entropy of certain binary sequences and
  applications--{I}.
\newblock {\em {IEEE} Trans. Inf. Theory}, 19(6):769--772, Nov. 1973.

\end{thebibliography}

\end{document}